\documentclass{conm-p-l}
\usepackage{amsmath, amsthm, amsfonts}
\usepackage{amssymb}
\usepackage{latexsym}
\usepackage{verbatim}
\usepackage{url}
\usepackage{pdflscape,multirow,array,bigstrut}

\usepackage[all,cmtip]{xy}
\usepackage{color}
\usepackage[usenames,dvipsnames]{xcolor}
\usepackage{hyperref}
\hypersetup{colorlinks=true,urlcolor=MidnightBlue,citecolor=PineGreen,linkcolor=BrickRed}

\usepackage{enumerate}
\usepackage{amssymb}
\usepackage{array}
\usepackage{nicefrac}
\usepackage[section]{placeins}

\DeclareMathAlphabet{\mathfr}{U}{euf}{m}{n}
\newtheorem{theorem}{Theorem}[section]

\newtheorem{proposition}[theorem]{Proposition}

\newtheorem{lemma}[theorem]{Lemma}

\theoremstyle{definition}

\newtheorem{remark}[theorem]{Remark}
\newtheorem{example}[theorem]{Example}

\numberwithin{equation}{section}

\newcommand{\Aut}{\operatorname{Aut}}
\newcommand{\Conj}{\operatorname{Conj}}

\newcommand{\End}{\operatorname{End}}
\newcommand{\Exp}{\mathrm{E}}
\newcommand{\Frob}{\operatorname{Frob}}
\newcommand{\Gal}{\mathrm{Gal}}

\newcommand{\Id}{\mathrm{Id}}

\newcommand{\ord}{\operatorname{ord}}
\newcommand{\Out}{\operatorname{Out}}

\newcommand{\Trace}{\operatorname{Trace}}
\newcommand{\Teich}{\operatorname{Teich}}
\newcommand{\lcm}{\operatorname{lcm}}

\newcommand{\GL}{\mathrm{GL}}

\newcommand{\SU}{\mathrm{SU}}
\newcommand{\Sym}{\mathrm{Sym}}

\newcommand{\Unitary}{\mathrm{U}}
\newcommand{\USp}{\mathrm{USp}}

\newcommand{\cyc}[1]{{\mathrm{C}_#1}}
\newcommand{\dih}[1]{{\mathrm{D}_#1}}

\newcommand{\rs}[2]{\left(\frac{#1}{#2}\right)}

\newcommand{\C}{\mathbb C}

\newcommand{\gM}{\mathfrak M}
\newcommand{\N}{\mathbb N}
\newcommand{\Q}{\mathbb Q}
\newcommand{\Qbar}{{\overline{\mathbb Q}}}

\newcommand{\Z}{\mathbb Z}

\newcommand{\nothing}{{}}
\newcommand{\Mtime}{\textsf{M}}
\newcommand{\fp}{\mathfrak{p}}

\begin{document}
\title{Sato-Tate groups of some weight 3 motives}
\date{\today}

\author{Francesc Fit\'e}
\address{Universit\"at Duisburg-Essen/Institut f\"ur Experimentelle Mathematik, Fakult\"at f\"ur Mathematik,
D-45127, Essen, Germany}
\curraddr{}
\email{francesc.fite@gmail.com}
\thanks{Fit\'e received financial support from the German Research council, via CRC 701.}

\author{Kiran S. Kedlaya}
\address{Department of Mathematics, University of California, San Diego, 9500 Gilman Drive \#0112, La Jolla, CA 92093, USA}
\curraddr{}
\email{kedlaya@ucsd.edu}
\urladdr{http://kskedlaya.org}
\thanks{Kedlaya was supported by NSF (grant DMS-1101343) and
UCSD (S.E. Warschawski professorship).}

\author{Andrew V. Sutherland}
\address{Department of Mathematics, Massachusetts Institute of Technology, 77 Mass. Ave., Cambridge, MA  02139, USA}
\curraddr{}
\email{drew@math.mit.edu}
\urladdr{http://math.mit.edu/~drew} %optional
\thanks{Sutherland was supported by NSF (grant DMS-1115455).}

\begin{abstract}
We establish the group-theoretic classification of Sato-Tate groups of self-dual motives of weight $3$ with rational coefficients and Hodge numbers
$h^{3,0} = h^{2,1} = h^{1,2} = h^{0,3} = 1$. We then describe families of motives that realize some of these
Sato-Tate groups, and provide numerical evidence supporting equidistribution. One of these families
arises in the middle cohomology of certain Calabi-Yau threefolds appearing in the Dwork quintic pencil; for motives
in this family, our evidence suggests that the Sato-Tate group is always equal to the full unitary symplectic group
$\USp(4)$.
\end{abstract}
\maketitle
\tableofcontents

\section{Introduction}\label{section: introduction}

For a fixed elliptic curve without complex multiplication defined over a number field, the \emph{Sato-Tate conjecture} predicts the average distribution of the Frobenius trace at a variable prime.
This conjecture 
may be naturally generalized to an arbitrary motive over a number field in terms of equidistribution
of classes within a certain compact Lie group, the \emph{Sato-Tate group},
as described in \cite[\S 13]{Ser95}, \cite[Ch.\ 8]{Ser12}, and \cite[\S 2]{FKRS12}. This equidistribution problem reduces naturally
(as described in \cite[Appendix to Chapter~1]{Ser68}) to establishing analytic properties of certain motivic
$L$-functions, but unfortunately this latter problem is generally quite difficult.
Besides cases of complex multiplication, one of the few cases where equidistribution is known is
elliptic curves over totally
real number fields \cite{BLGG11}.

However, the problem of classifying the Sato-Tate groups  that can arise from a given class of motives is more tractable. This problem splits naturally into two subproblems: the \emph{group-theoretic classification} problem
of identifying those groups consistent with certain group-theoretic restrictions known to apply to Sato-Tate groups
in general, and the \emph{arithmetic matching} problem of correlating the resulting groups with 
the arithmetic of motives in the family. In the case of 1-motives 
of abelian surfaces, both subproblems have been solved in \cite{FKRS12}: there turn out to be exactly
52 groups that arise, up to conjugation within the unitary symplectic group $\USp(4)$.

In this paper, we consider a different family of motives for which we solve the group-theoretic classification
problem, give some partial results towards the arithmetic matching problem, and present numerical evidence supporting the equidistribution conjecture. Before describing the family of motives in question, let us 
recall the general formulation of the group-theoretic classification problem for
self-dual motives with rational coefficients of fixed weight $w$, dimension $d$, and Hodge numbers $h^{p,q}$. The problem is to identify groups
obeying the \emph{Sato-Tate axioms}, as formulated in \cite{FKRS12} (modulo one missing condition; see
Remark~\ref{R:why not unitary2}).
\begin{enumerate}
\setlength\itemindent{10pt}
\item[(ST1)]
The group $G$ is a closed subgroup of $\USp(d)$ or $\operatorname{O}(d)$, depending on whether $w$ is odd or even (respectively).
\item[(ST2)] (Hodge condition)
There exists a subgroup $H$ of $G$, called a \emph{Hodge circle}, which is the image of 
a homomorphism $\theta\colon \Unitary(1) \to G^0$
such that $\theta(u)$ has eigenvalues $u^{p-q}$ with multiplicity $h^{p,q}$.
Moreover, the Hodge circles generate a dense subgroup of the identity component~$G^0$.
\item[(ST3)] (Rationality condition)
For each component $C$ of $G$ and each irreducible character $\chi$ of $\GL_{d}(\C)$, the expected value (under the Haar measure) of $\chi(\gamma)$ over $\gamma\in C$ is an integer.
\end{enumerate}
For fixed $w,d,h^{p,q}$, there are only finitely many groups $G$ satisfying (ST1), (ST2), and (ST3),
up to conjugation within $\USp(d)$ or $\operatorname{O}(d)$; see Remark 3.3 in \cite{FKRS12}.

Since the group-theoretic classification is known for 1-motives of abelian varieties of dimensions $1$ and $2$, it is natural
to next try the case of abelian threefolds. We are currently working on this classification, but it is likely
to be rather complicated, involving many hundreds of groups. In this paper, we instead consider the case where
$w=3$, $d=4$, and $h^{3,0} = h^{2,1} = h^{1,2} = h^{0,3} = 1$. We have chosen this case because, on the one hand,
it is similar enough to the case of abelian surfaces that much of the analysis of \cite{FKRS12} carries over,
and, on the other hand, it is of some arithmetic interest due to the multiple ways in which such motives arise.
One of these ways is by taking the symmetric cube of the 1-motive associated to an elliptic curve.
Another way is to consider a member of the \emph{Dwork pencil} of Calabi-Yau projective threefolds
defined by the equation
\begin{equation}\label{eq:quintic}
x_0^5 + x_1^5 + x_2^5 + x_3^5 + x_4^5 = t x_0 x_1 x_2 x_3 x_4,
\end{equation}
in which $t$ represents a nonzero parameter, and then extract the 3-motive invariant under the action of the automorphism group
$(\Z/5\Z)^4$. These two constructions are closely related: for instance, the coincidence between certain mod $\ell$ Galois representations arising from the two constructions is exploited in \cite{HSBT10} to yield one of the key ingredients in the proof of the Sato-Tate conjecture for elliptic curves. Additional constructions
can be achieved using direct sums and tensor products of motives associated to elliptic curves and modular forms (the latter case was suggested to us by Serre).

The primary result of this paper is the resolution of the group-theoretic classification problem for motives
of the shape we have just described. This turns out to be similar to the classification problem in \cite{FKRS12}
but substantially simpler due to the less symmetric shape of the Hodge circle: we end up with only 26 groups up to conjugation.
These groups are described in \S\ref{section: classification} and summarized in
Table~\ref{table:STgroups}. As in \cite{FKRS12},
we compute moment sequences associated to these groups in order to facilitate numerical experiments; these appear in \S\ref{section: moments}.

As a partial result towards the arithmetic matching problem, we describe several constructions yielding motives of the given form and then match examples of these constructions to our list of Sato-Tate groups based on numerical experiments.
For example, the symmetric cube construction gives rise to Sato-Tate groups with identity component $\Unitary(1)$
or $\SU(2)$, depending on whether or not the original elliptic curve has complex multiplication (CM), and we can provisionally
identify the exact Sato-Tate group (up to conjugation) by comparing experimentally derived moment statistics with the moment sequences computed in \S\ref{section: moments}.
In the CM case we are actually able to prove equidistribution using the techniques developed in \cite{FS12}; this follows from Lemma \ref{lemma:CMproddist}.
More generally, using the direct sum of a pair of motives arising from CM modular forms of weights 2 and 4, we obtain examples matching all 10 of the groups in our classification that have identity component $\Unitary(1)$, and we are able to prove equidistribution in each of these cases (see Lemma \ref{lemma:corrmf}).
Additional cases arise from considering Hilbert modular forms and Hecke characters over CM fields.
In total, we exhibit examples that appear to realize 25 of the 26 possible Sato-Tate groups obtained by our classification.

For the Dwork pencil construction, we are able to collect numerical evidence thanks to the work of 
Candelas, de la Ossa, and Rodriguez Villegas 
\cite{COR00, COR03}, who, motivated by the appearance of the Dwork pencil in the study of \emph{mirror symmetry}
in mathematical physics, described some $p$-adic analytic formulas for the $L$-function coefficients.
The resulting evidence may be a bit surprising on first glance:
one might expect (by analogy with abelian varieties) that
the group $\USp(4)$ arises for most members of the pencil with a sparse but infinite set of exceptions, but in fact we found no exceptions at all other than $t=0$ (the Fermat quintic). A Hodge-theoretic heuristic suggesting the existence of only finitely many exceptions in this family (and also applicable in many other cases) has been proposed by de Jong \cite{dJ02}.

For a gentle introduction to motives, we refer the reader to \cite{Milne}. 

\section{Group-theoretic classification}\label{section: classification}

In this section, we classify, up to conjugation, the groups $G \subseteq \GL_4(\C)$ that satisfy the Sato-Tate axioms (ST1), (ST2), and (ST3); the list of possible groups (in notation introduced
later in this section) can be found in Table~\ref{table:STgroups}.
As in \cite{FKRS12},
we exhibit explicit representatives of each conjugacy class for the purposes of computing moments, which are
needed for our numerical experiments (see \S\ref{section: moments}).
 This forces us to give an explicit description of the matrix groups we are using.

Let $M$ (resp. $S$) denote a matrix of $\GL_4(\C)$ corresponding to a Hermitian (resp. symplectic) form, that is, a matrix satisfying $M^t=\overline M$ (resp. $S^t=-S$). The unitary symplectic group of degree $4$ (relative to the forms $M$ and $S$) is defined as
\[
\USp(4):=\left\{A\in\GL_4(\C)\,|\,A^tSA=S\,,\,\overline A^tMA=M\,\right\}\,.
\]
For the purposes of the classification, it will be convenient to make different choices of $S$ and $M$ according to the different possibilities for the identity component $G^0$ of $G$. Unless otherwise specified, we will take~$M$ to be the identity matrix $\Id$.

As in \cite[Lemma~3.7]{FKRS12}, one shows that if $G$ satisfies the Sato-Tate axioms,
then $G^0$ is conjugate to one of
\[
\Unitary(1),\medspace \SU(2),\medspace \Unitary(2), \medspace \Unitary(1) \times \Unitary(1),\medspace \Unitary(1) \times \SU(2),\medspace \SU(2) \times \SU(2),\medspace \USp(4).
\] 
(The case $\Unitary(2)$ does not occur in \cite[Lemma~3.7]{FKRS12}; see Remark~\ref{R:why not unitary2} for the reason why.)
We now proceed by considering each of these options in turn. Throughout the discussion, let
$Z$ and $N$ denote the centralizer and normalizer, respectively, of $G^0$ in $\USp(4)$, so that
$N/(ZG^0)$ is finite and $G \subseteq N$. (Beware that this convention is followed in \cite[\S 3.4]{FKRS12} but not in \cite[\S 3.5]{FKRS12}.)

\subsection{The case $G^0=\Unitary(1)$}\label{section: unitary}

To treat the case $G^0 = \Unitary(1)$, 
we assume that the symplectic form preserved by $\USp(4)$ is given by the matrix
\[
S := \begin{pmatrix} 0 & \Id_2 \\
-\Id_2 & 0
\end{pmatrix}\,.
\]
In this case $G^0$ must be equal to a Hodge circle $H$, which we may take to be the image of the homomorphism
\begin{equation}\label{equation: HodgeU1}
\theta\colon \Unitary(1) \to \USp(4)\,,\qquad \theta(u) :=\begin{pmatrix} U & 0 \\ 0 & \overline U \end{pmatrix}\,,\qquad U:=\begin{pmatrix}u^3 & 0\\ 0 & u\end{pmatrix}\,.
\end{equation}
Note that the centralizer of $G^0$ within $\GL(4, \C)$ consists of diagonal matrices. For such a matrix to be symplectic and unitary it must be of the form
%\[
%\begin{pmatrix} v_1 & 0 & 0 & 0 \\ 0 & v_2 & 0 & 0 \\ 0 & 0 & \overline v_1 & 0 \\ 0 & %0 & 0 & \overline v_2  \end{pmatrix}\,,
%\]
\begin{equation}\label{equation: centralizerU1}
\begin{pmatrix} V_2 & 0 \\ 0 & \overline V_2  \end{pmatrix}\,,\qquad V_2 :=\begin{pmatrix} v_1 & 0 \\ 0 & v_2
\end{pmatrix}\,,
\end{equation}
where $v_1$ and $v_2$ are in $\Unitary(1)$. We thus conclude that $Z \simeq \Unitary(1)\times\Unitary(1)$.
The quotient $N/Z$ injects into the continuous automorphisms $\Aut^{\text{cont}}(G^0)$ of $G^0$. Since $\Aut^{\text{cont}}(\Unitary(1))$ consists just of the identity and complex conjugation, $Z$ has index 2 in $N$. Thus $N$ has the form
\[
N = Z \cup JZ, \qquad J := \begin{pmatrix} 0 & J_2 \\ -J_2 & 0 \end{pmatrix}\,,\qquad J_2:=\begin{pmatrix}
1 & 0 \\ 0 & -1 
\end{pmatrix}\,.
\]
Conjugation on $Z$ by $J$ corresponds to complex conjugation, thus we have
\[
N/G^0\medspace \simeq \medspace \Unitary(1) \rtimes \Z/2\Z \,,
\]
where the nontrivial element of $\Z/2\Z$ acts on $\Unitary(1)$ by complex conjugation.

We first enumerate the options for $G$ assuming that $G\subseteq Z$.
Any finite subgroup of order $n$ of $Z/G^0\simeq \Unitary(1)$ is cyclic. It lifts to a subgroup $C_n$ of $Z$, for which we may choose the following presentation:
$$
C_n:=\langle G^0,\zeta_n\rangle\,,\qquad \zeta_n:=\begin{pmatrix}
\Theta_n & 0 \\ 0 & \overline \Theta_n
\end{pmatrix}\,,
\qquad
\Theta_n :=\begin{pmatrix} e^{2\pi i/n} & 0 \\ 0 & 1 \end{pmatrix}\,.
$$

\begin{lemma}
If the rationality condition (ST3) is satisfied for $C_n$, then $n$ lies in $\{1,2,3,4,6\}$.
\end{lemma}

\begin{proof}
By the rationality condition, the average over $r\in[0,1]$ of the fourth power of the trace of the matrix
$$
\theta(e^ {2\pi i r})\zeta_n
$$
is an integer. It is an elementary but tedious computation to check that this average is equal to
\[
36+8\cos\left(\frac{2\pi}{n}\right)\,.
\]
This implies $\cos\left(\frac{2\pi}{n}\right)=\frac{i}{2}$, for $i\in\{-2,-1,0,1,2\}$, hence $n\in\{1,2,3,4,6\}$.
%\[
%2^2\left(11-4\sin^2\left(\frac{2\pi}{n}\right)\right)\,.
%\]
%This implies that $\sin^2\left(\frac{2\pi}{n}\right)=\frac{i}{4}$ for %$i\in\{0,1,2,3,4\}$. Thus 
%$$
%\sin\left(\frac{2\pi}{n}\right)\in\left\{0,\pm\frac{1}{2},\pm\frac{1}{\sqrt 2}, %\pm\frac{\sqrt 3}{2}\right\}\,,
%$$
%which implies the statement of the lemma.
\end{proof}

We now consider the case $G\not\subseteq Z$. For $n\in\{1,2,3,4,6\}$, define 
\[
J(C_n):=\langle G^0,\zeta_n, J\rangle\,.
\]

%\begin{remark}
%Observe that conjugation by $J$ and $\zeta_n$ for $n\in\{1,2,3,4,6\}$ preserves  (the conjugacy class in $G^0$ of) %the Hodge circle.
%\end{remark}

\begin{lemma} Let $G$ be a subgroup of $N$ satisfying the rationality condition (ST3), and for which $\theta(\Unitary(1))\subseteq G\not\subseteq Z$.  Then $G$ is conjugate 
to $J(C_n)$ for some $n\in\{1,2,3,4,6\}$.
\end{lemma}

\begin{proof}
By hypothesis, $G$ contains an element of $JZ$, which is of the form
\[
JV=\begin{pmatrix} 0 & J_2V_2 \\ -J_2\overline V_2 & 0 \end{pmatrix}\,,\qquad \text{where }\quad J_2V_2=\begin{pmatrix}
v_1 & 0 \\ 0 & -v_2 
\end{pmatrix}\,,
\]
where $v_1$ and $v_2$ are in $\Unitary(1)$. The conjugate of $JV$ by the matrix
\[
W:=\begin{pmatrix} 0 & W_2 \\ -\overline W_2 & 0 \end{pmatrix}\,,\qquad W_2:=\begin{pmatrix}
-\sqrt {v_1} & 0 \\ 0 & \sqrt{v_2} 
\end{pmatrix}
\]
is $J$. Thus the conjugate of $G$ by $W$ is of the form $H\rtimes \langle J\rangle$, where $H$ is a subgroup of $Z$ satisfying the rationality condition. As we have already seen, $H$ must be equal to $C_n$ for some $n\in\{1,2,3,4,6\}$.
\end{proof}

\subsection{The case $G^0=\SU(2)$}

To treat the case $G^0=\SU(2)$, 
we consider the standard representation of $\SU(2)$ on $\C^2$
and take the embedding of $\SU(2)$ in $\USp(4)$ corresponding to the representation $\operatorname{Sym}^ 3(\C^ 2)$. More explicitly, if $a,b\in\C$ are such that $a\overline a+ b\overline b= 1$, we consider the embedding of $\SU(2)$ in $\USp(4)$ given by
\begin{equation}\label{equation: embedSU2}
\begin{pmatrix}
a & b\\-\overline b & \overline a
\end{pmatrix}
\mapsto
\begin{pmatrix}
a^3 & a^2b & ab^2 & b^3\\
-3a^2\overline b & a^2\overline a-2 a b\overline b & 2a\overline a b- b^2\overline b & 3 \overline a b^2 \\
3a\overline b^2 & b\overline b^2-2a\overline a \overline b & a\overline a^2-2\overline a b \overline b & 3 \overline a^2 b \\
-\overline b^3 & \overline a\overline b^2 & -\overline a^2\overline b & \overline a^3 
\end{pmatrix}.
\end{equation} 
In this section, the Hodge circle is the image of the homomorphism
\begin{equation}\label{equation: HodgeSU2}
\theta\colon \Unitary(1) \to \USp(4)\,,\qquad \theta(u) :=\begin{pmatrix} U & 0 \\ 0 & \overline u^4 U \end{pmatrix}\,,\qquad U:=\begin{pmatrix}u^3 & 0\\ 0 & u\end{pmatrix}\,,
\end{equation} and we assume that the symplectic and Hermitian forms preserved by $\USp(4)$ are respectively given by the matrices
\[
S := \begin{pmatrix} 0 & 0 & 0 & z \\
0 & 0 & -1/z & 0\\
0 & 1/z & 0 & 0\\
-z & 0 & 0 & 0
\end{pmatrix}\,,
\qquad
M := \begin{pmatrix} 1/z & 0 & 0 & 0 \\
0 & z & 0 & 0\\
0 & 0 & z & 0\\
0 & 0 & 0 & 1/z
\end{pmatrix}\,,
\]
where $z=\sqrt 3$. Since the embedded $\SU(2)$ contains the embedded $\Unitary(1)$ of the previous section, the centralizer $Z$ of $G^0$ in $\USp(4)$ consists of matrices of the form \eqref{equation: centralizerU1}. Imposing the condition that conjugation by such a matrix preserves any element of the embedded $\SU(2)$, one finds that $v_1=v_2=\overline v_1=\overline v_2$. Thus $Z=\{\pm \Id\} \subseteq G^0$. The group $N/G^0 = N/(ZG^0)$
embeds into the group of continuous outer automorphisms $\Out^{\text{cont}}(\SU(2))$, which is trivial;
consequently, this case yields only the single group $D := G^0$.

%As in the previous case, $G^0Z=G^0$ has index 2 in the normalizer $N$ of $G^0$ in %$\USp(4)$, and it is then elementary that it has the form
%\[
%N = G^0 \cup JG^0, \qquad J := \begin{pmatrix} 0 & J_2 \\ J_2 & 0 \end{pmatrix}\,,%\qquad J_2:=\begin{pmatrix}
%0 & 1 \\ -1 & 0 
%\end{pmatrix}\,.
%\]

\subsection{The case $G^0=\Unitary(2)$}\label{section: Unitary(2)}

To treat the case $G^0 = \Unitary(2)$, 
we again assume that the symplectic form preserved by $\USp(4)$ is given by the matrix
\[
S := \begin{pmatrix} 0 & \Id_2 \\
-\Id_2 & 0
\end{pmatrix}\,.
\]
The group $\Unitary(2)$ embeds into $\USp(4)$ via the map given in block form by
\[
A \mapsto \begin{pmatrix} A & 0 \\ 0 & \overline{A} \end{pmatrix},
\]
as in \cite[(3.1)]{FKRS12}. As indicated in \cite[\S 3]{FKRS12}, we have $Z=\{\pm \Id\} \subseteq G^0$ 
and $N = \Unitary(2) \cup J(\Unitary(2))$ for 
\[
J := \begin{pmatrix} 0 & J_2 \\ -J_2 & 0 \end{pmatrix}, \qquad
J_2 := \begin{pmatrix} 0 & 1 \\ -1 & 0 \end{pmatrix}.
\]
We thus obtain two groups: $\Unitary(2)$ and $N(\Unitary(2))$.

\begin{remark} \label{R:why not unitary2}
Note that $\Unitary(2)$ is missing from 
\cite[Theorem~3.4]{FKRS12} even though it satisfies the Sato-Tate axioms as formulated in
\cite[Definition~3.1]{FKRS12}. The reason is that axiom (ST2) is stated incorrectly there:
it fails to include the condition that the Hodge circles generate a dense subgroup of $G^0$; see \cite[8.2.3.6(i)]{Ser12}. 

Let us see this point more explicitly. 
Let $\theta\colon \Unitary(1) \to \Unitary(2)$ be a continuous homomorphism. 
The map
$\Unitary(1) \times \SU(2) \to \Unitary(2)$ taking $(u,A)$ to $u A$ is an isogeny of degree $2$
with kernel generated by $(-1, -\Id_2)$.
We may thus identify $\Unitary(2)/\SU(2)$ with $\Unitary(1)/\{\pm 1\}$
and then with $\Unitary(1)$ via the squaring map. There must then exist an integer $a$
such that for all $u \in \Unitary(1)$, the image of $\theta(u)$ in $\Unitary(1)$ is $u^a$.
The formula $u \mapsto u^{-a} \theta(u)^2$ defines a homomorphism $\Unitary(1) \to \SU(2)$, so there must exist
an integer $b$ such that for all $u \in \Unitary(1)$, 
the image of $u \in \Unitary(1)$ in $\SU(2)$ has eigenvalues $u^b$ and $ u^{-b}$.
The eigenvalues of $\theta(u^2)$ must then be $u^{a+2b}$ and $u^{a-2b}$.
If we then embed $\Unitary(2)$ into $\USp(4)$, the image of $\theta(u^2)$ has eigenvalues
$u^{a+2b}, u^{a-2b}, u^{-a+2b}, u^{-a-2b}$. 

In this paper, we get a Hodge circle by taking $\theta$ as above with $a = 4, b = 1$.
By contrast, in the setting of \cite{FKRS12},
the eigenvalues must be $u^2, u^2, u^{-2}, u^{-2}$, in some order.  We may assume without loss of generality
that $a+2b = 2$; we must then have $a - 2b \in \{-2, 2\}$, implying that either $a=0$ or $b=0$.
If $a=0$, then the conjugates of the image of $\theta$ all lie inside $\SU(2)$, and if $b=0$,
then the conjugates all lie inside $\Unitary(1)$. Thus no Hodge circle can exist.
\end{remark}

\subsection{The remaining cases for $G^0$}

We now treat the remaining cases for $G^0$. These turn out to give exactly the same answers as in
\cite[\S 3.6]{FKRS12}, modulo the position of the Hodge circle, which we will ignore
(see Remark~\ref{remark:embeddings of Hodge circles}); it thus suffices to recall these answers briefly.
The case $G^0=\USp(4)$ is trivial, so we focus on the split cases.
As in \cite[\S 3.6]{FKRS12},
we assume that the symplectic form preserved by $\USp(4)$ is defined by the block matrix
\[
S := \begin{pmatrix} S_2 & 0 \\
0 & S_2
\end{pmatrix}\,,
\qquad
S_2:= \begin{pmatrix} 0 & 1 \\
-1 & 0
\end{pmatrix}\,,
\]
and that product groups are embedded compatibly with this decomposition of the symplectic form.

For $G^0=\SU(2)\times\SU(2)$, as in \cite[\S 3.6]{FKRS12} we have the group $G_{3,3}:=G^0$ itself and its normalizer $N(G_{3,3})$, obtained by adjoining to $G^0$ the matrix
\[
\begin{pmatrix} 0 & S_2 \\ -S_2 & 0 \end{pmatrix}\,.
\]

For $G^0=\Unitary(1)\times\Unitary(1)$, the normalizer in $\USp(4)$ contains $\Unitary(1) \times \Unitary(1)$ with index~8, and the quotient is isomorphic to the dihedral group $\dih 4$ and generated by
matrices
\[
a := \begin{pmatrix} S_2 & 0 \\ 0 & \Id_2 \end{pmatrix}\,, \qquad
b := \begin{pmatrix} \Id_2 & 0 \\ 0 & S_2 \end{pmatrix}\,, \qquad
c := \begin{pmatrix} 0 & \Id_2 \\ -\Id_2 & 0 \end{pmatrix}\,,
\]
each of which defines an involution on the component group. We write $F_S$
for the group generated by $G^0$ and a subset $S$ of $\langle a,b,c\rangle$. As in \cite[\S 3.6]{FKRS12}, up to conjugation we obtain eight groups
\[
F_\nothing, F_a, F_c,
F_{a,b}, F_{ab}, F_{ac}, F_{ab,c}, F_{a,b,c}.
\]
For $G^0=\Unitary(1)\times\SU(2)$, we obtain the group
$G_{1,3} := \Unitary(1) \times \SU(2)$ and its normalizer $N(G_{1,3})= \langle G_{1,3}, a\rangle$.

%nonconjugate embeddings of $G^0$ in $\USp(4)$ if one keeps track of the Hodge circle. With the notation of (\ref{equation: embed U(1)xSU(2)}), they are respectively given by sending $(u,A)\in\Unitary(1)\times\SU(2)$ to the matrices
%\[
%\begin{pmatrix} U & 0 \\ 0 & A \end{pmatrix}, \qquad
%\begin{pmatrix} A & 0 \\ 0 & U^3 \end{pmatrix}\,.
%\]
%We denote the corresponding groups by $G_1$ and $G_2$. The normalizers of these groups are $G_{1,a}:=\langle G_1, a\rangle$ and $G_{2,b}:=\langle G_2, b\rangle$. 

%\begin{remark} Among the groups encountered in this section, only $E_1$, $F_\nothing$, $F_{ab}$, $G_1$, and $G_2$ %satisfy the condition that the action of the group of components preserves the (conjugacy class of) the Hodge %circle.
%\end{remark}

%
%We will take $H$ to be the image of the homomorphism
%\begin{equation}\label{equation: embed U(1)xSU(2)}
%\theta: \Unitary(1) \to \USp(4)\,,\qquad \theta(u) :=\begin{pmatrix} U & 0 \\ 0 & U^3 \end{pmatrix}\,,\qquad U:=\begin{pmatrix}u & 0\\ 0 & \overline u\end{pmatrix}\,.
%\end{equation}
%We will take $\Unitary(1)$ embedded in $\SU(2)$ by mapping $u\in\Unitary(1)$ to the matrix $U$ and we will consider $\SU(2)$ embedded in block form in $\USp(4)$   
%\[
%A\mapsto \begin{pmatrix} A & 0 \\ 0 & \overline A \end{pmatrix}\,.
%\]
%We will 
%assume that the symplectic form preserved by $\USp(4)$ is defined by the block matrix
%\[
%S := \begin{pmatrix} S_2 & 0 \\
%0 & S_2
%\end{pmatrix}\,,
%\qquad
%S_2:= \begin{pmatrix} 0 & 1 \\
%-1 & 0
%\end{pmatrix}\,.
%\]

\begin{remark} \label{remark:embeddings of Hodge circles}
Note that in some of the cases with $G^0 = \Unitary(1) \times \Unitary(1)$, there is more than one way to
embed the Hodge circle $H$ into $G$ up to conjugation. This is irrelevant for questions of equidistribution, but
it does matter when one attempts to relate the Sato-Tate group of a motive with the real endomorphism algebra of its Hodge structure (as in \cite[\S 4]{FKRS12}). Since we will not attempt that step in this paper at more than a heuristic level, we have chosen to ignore this ambiguity.
\end{remark}

\section{Testing the generalized Sato-Tate conjecture}\label{section: moments}

In the sections that follow, we describe various explicit constructions that give rise to self-dual 3-motives with Hodge numbers $h^{3,0} = h^{2,1} = h^{1,2} = h^{0,3} = 1$ and rational coefficients.
For each of these motives $M$, we then perform numerical tests of the generalized Sato-Tate conjecture by comparing
the distribution of the normalized $L$-polynomials of $M$ with the distribution of characteristic polynomials in one of the candidate Sato-Tate groups $G$ found by the classification in~\S\ref{section: classification}.
More precisely, we ask whether the normalized $L$-polynomials  of $M$ appear to be equidistributed with respect to the image of the Haar measure under the map $G \to \Conj(\USp(4))$, where $\Conj$ denotes the space of conjugacy classes.
To make this determination, we compare \emph{moment statistics} of the motive~$M$ to \emph{moment sequences} associated to~$G$,  as described below.

Table \ref{table:STgroups} lists invariants that allow us to distinguish the groups $G$. As in \cite{FKRS12}, $d$ denotes the real dimension of $G$; $c$ is the number $|G/G^0|$ of connected components of $G$; and $z_1$ and $z_2$ are defined by
$$
z_1:=z_{1,0}
\,,\qquad 
z_2:=[z_{2,-2},z_{2,-1},z_{2,0},z_{2,1},z_{2,2}]\,,
$$ 
where $z_{i,j}$ is the number of connected components of $G$ for which the $i$th coefficient $a_i$ of he characteristic polynomial of each of its elements is equal to the integer $j$.  We use $[G/G^0]$ to denote the isomorphism class of the component group of $G$, and the notations  $\cyc n$ and $\dih n$ indicate the cyclic group of $n$ elements and the dihedral group of $2n$ elements, respectively.
For each of the motives $M$ constructed in the sections that follow, the nature of the construction allows us to predict the type of identity component and the number of components, as well as the values of the invariants $z_1$ and~$z_2$, which is enough to uniquely determine a candidate Sato-Tate group $G$.
The last column of Table~\ref{table:STgroups} references the example motives $M$ whose candidate Sato-Tate group is $G$.
For all but one group ($F_{a,b,c}$) there is at least one such example, and in many cases there are multiple constructions that lead to the same candidate Sato-Tate group.

\begin{table}
\begin{center}
\setlength{\extrarowheight}{0.5pt}
\caption{Candidate Sato-Tate groups of self-dual motives of weight 3 with
Hodge numbers $h^{3,0} = h^{2,1} = h^{1,2} = h^{0,3} = 1$ and rational coefficients. The final column indicates where within the article
to find explicit constructions that yield matching moment statistics.} \label{table:STgroups}
\vspace{12pt}
\begin{tabular}{rrllrrr}
$d$ & $c$  & $G$          & $[G/G^0]$              &$z_1$ & $\qquad\qquad z_2$ & Examples \\\hline\vspace{-10pt}\\
$1$ & $1$  & $C_1$        & $\cyc{1}$              &  $0$ & $0,0,0,0,0$ & \ref{example:mfsum_C1}, \ref{example:eccube_C1}, \ref{example:mfprod_C1} \\
$1$ & $2$  & $C_2$        & $\cyc{2}$              &  $0$ & $0,0,0,0,0$ & \ref{example:mfsum_C2} \\
$1$ & $3$  & $C_3$        & $\cyc{3}$              &  $0$ & $0,0,0,0,0$ & \ref{example:mfsum_C3}, \ref{example:ecprod_C3}\\
$1$ & $4$  & $C_4$        & $\cyc{4}$              &  $0$ & $0,0,0,0,0$ & \ref{example:mfsum_C4} \\
$1$ & $6$  & $C_6$        & $\cyc{6}$              &  $0$ & $0,0,0,0,0$ &  \ref{example:mfsum_C6} \\
$1$ & $2$  & $J(C_1)$     & $\dih{1}$              &  $1$ & $0,0,0,0,1$ & \ref{example:mfsum_C1}, \ref{example:eccube_C1}, \ref{example:mfprod_C1} \\
$1$ & $4$  & $J(C_2)$     & $\dih{2}$              &  $2$ & $0,0,0,0,2$ & \ref{example:mfsum_C2} \\
$1$ & $6$  & $J(C_3)$     & $\dih{3}$              &  $3$ & $0,0,0,0,3$ & \ref{example:mfsum_C3}\\
$1$ & $8$  & $J(C_4)$     & $\dih{4}$              &  $4$ & $0,0,0,0,4$ & \ref{example:mfsum_C4} \\
$1$ & $12$ & $J(C_6)$     & $\dih{6}$              &  $6$ & $0,0,0,0,6$ &  \ref{example:mfsum_C6} \\
$3$ & $1$ & $D$           & $\cyc 1$               &  $0$ & $0,0,0,0,0$ & \ref{example:eccube_C1} \\
$4$ & $1$ & $\Unitary(2)$ & $\cyc 1$               &  $0$ & $0,0,0,0,0$ & \ref{example:ecprod_U2}, \ref{example:mfprod_U2} \\
$4$ & $2$ & $N(\Unitary(2))$ & $\cyc 2$            &  $1$ & $0,0,0,0,0$ &\ref{example:mfprod_U2}\\ 
$2$ & $1$  & $F_\nothing$ & $\cyc{1}$              &  $0$ & $0,0,0,0,0$ &\ref{example:mfsum_F}, \ref{example:ecprod_F}, \ref{example:mfprod_F} \\
$2$ & $2$  & $F_a$        & $\cyc{2}$              &  $0$ & $0,0,0,0,1$ & \ref{example:mfsum_F} \\
$2$ & $2$  & $F_c$        & $\cyc{2}$              &  $1$ & $0,0,0,0,0$ & \ref{example:mfprod_F}, \ref{example:ecprod_F} \\ 
$2$ & $2$  & $F_{ab}$     & $\cyc{2}$              &  $1$ & $0,0,0,0,1$ & \ref{example:ecprod_F} \\
$2$ & $4$  & $F_{ac}$     & $\cyc{4}$              &  $3$ & $0,0,2,0,1$ & \ref{example:Fac}\\
$2$ & $4$  & $F_{a,b}$    & $\dih{2}$              &  $1$ & $0,0,0,0,3$ & \ref{example:mfprod_F}, \ref{example:ecprod_F} \\
$2$ & $4$  & $F_{ab,c}$   & $\dih{2}$              &  $3$ & $0,0,0,0,1$ &\ref{example:mfprod_F} \\ 
$2$ & $8$  & $F_{a,b,c}$  & $\dih{4}$              &  $5$ & $0,0,2,0,3$ & None (but see \ref{example:hecke_char})\\  
$4$ & $1$  & $G_{1,3}$    & $\cyc{1}$              &  $0$ & $0,0,0,0,0$ & \ref{example:mfsum_G13} \\
$4$ & $2$  & $N(G_{1,3})$ & $\cyc{2}$              &  $0$ & $0,0,0,0,1$ & \ref{example:mfsum_G13} \\
$6$ & $1$  & $G_{3,3}$    & $\cyc{1}$              &  $0$ & $0,0,0,0,0$ & \ref{example:mfsum_G33} \\
$6$ & $2$  & $N(G_{3,3})$ & $\cyc{2}$              &  $1$ & $0,0,0,0,0$ & 
\ref{example:hmf}\\
$10$& $1$  & $\USp(4)$    & $\cyc{1}$              &  $0$ & $0,0,0,0,0$ & \ref{example:USp4} \\
\hline
\end{tabular}
\end{center}
\end{table}

\subsection{Experimental methodology --- moment statistics}

All of the motives $M/K$ that we consider have $L$-polynomials of the form
\begin{equation}\label{eq:Lpoly}
L_{\fp}(T) = p^6T^4 + c_1p^3T^3 + c_2pT^2+c_1T + 1,
\end{equation}
where $\fp$ is a prime of $K$ of good reduction for $M$, $p=N(\fp)$ is its absolute norm, and $c_1$ and $c_2$ are integers satisfying the Weil bounds $|c_1|\le 4p^{3/2}$ and $|c_2| \le 6p^2$ (in fact $c_2 \ge -2p^2$).
For the purpose of computing moment statistics we may restrict our attention to primes $\fp$ of degree 1, so we assume that $p$ is prime.
Note that $c_1$ is the \emph{negation} of the trace of Frobenius, and $c_2$ is obtained by \emph{removing a factor} of $p$ from the coefficient of $T^2$ in $L_\fp(T)$.

The normalized $L$-polynomial coefficients of $M/K$ are then defined by
\begin{equation}\label{eq:a1a2}
a_1(\fp):=c_1/N(\fp)^{3/2}\qquad\text{and}\qquad a_2(\fp):= c_2/N(\fp)^2,
\end{equation}
which are real numbers in the intervals $[-4,4]$ and $[-2,6]$, respectively.

Given a norm bound $B$, we let $S(B)$ denote the set of degree 1 primes of $K$ with norm at most $B$, and for $i=1,2$ we define the $n$th \emph{moment statistic} of $a_i$ for the motive $M$ (with respect to $B$) by
\[
M_n[a_i] := \frac{1}{\#S(B)} \sum_{\fp\in S(B)} a_i(\fp)^n.
\]
Similarly, given a candidate Sato-Tate group $G$, we let $a_i:=a_i(g)$ denote the $i$th  coefficient of the characteristic polynomial of a random element $g$ of $G$ (according to the Haar measure).  We then let $M_n[a_i]$ denote the expected value of $a_i^n$; this is the $n$th \emph{moment} of $a_i$ for the group $G$, which is always an integer (see axiom (ST3) in \cite[Def.\ 3.1]{FKRS12}).
In what follows it will be clear from context whether $M_n[a_i]$ refers to a moment statistic of $M$ (with respect to a norm bound $B$) or a moment of $G$.

To test for equidistribution with respect to a candidate Sato-Tate group~$G$, for increasing values of $B$ we  compare moment statistics $M_n[a_i]$ for the motive~$M$ to the corresponding moments $M_n[a_i]$ of the group $G$ and ask whether the former appear to converge to the latter as $B$ increases.
As may be seen in the tables of moment statistics listed in \S\ref{section: tables}, in cases where it is computationally feasible to make $B$ sufficiently large (up to $2^{40}$), we see very strong evidence for convergence; the moment statistics of $M$ generally agree with the moments of~$G$ to within one part in ten thousand.

It should be noted that the correct statement of the generalized Sato-Tate conjecture is somewhat more precise than what we are testing here. It includes both
a defined group $G$ attached to the motive (the \emph{Sato-Tate group}) and a sequence of elements of $\Conj(G)$ 
that should be equidistributed
for the image of the Haar measure, even before projecting to $\Conj(\USp(4))$. 
The formulation in \cite[\S 2]{FKRS12} is only valid for motives of weight 1; for a reformulation in terms of absolute
Hodge cycles that applies to motives of any odd weight, see \cite{BK,BK2}.

Since we do not introduce the definition of the Sato-Tate group here, we do not attempt to verify in our examples that the candidate Sato-Tate group we identify actually coincides with the Sato-Tate group of the motive. It is unclear how difficult this is to achieve, especially for the motives appearing in the Dwork pencil. Moreover, we do not claim that our list of constructions is exhaustive.  It may (or may not) be that the group $N(G_{3,3})$, which we are unable to match with an explicit construction, can be realized by other methods
(compare Remark~\ref{remark:abelian surface groups}).

\subsection{Moment sequences of candidate Sato-Tate groups}

In this section we compute moment sequences associated to each of the subgroups $G$ of $\USp(4)$ encountered in \S\ref{section: classification}; these are listed in Tables~\ref{table:a1moments} and~\ref{table:a2moments}.
Let~$G$ be a compact subgroup of $\USp(4)$. For $i=1,2$, let $a_i:=a_i(g)$ denote the $i$th coefficient of the characteristic polynomial of a random element $g$ of $G$ (according to the Haar measure). For a nonnegative integer $n$, the $n$th moment $M_n[a_i]$ is the expected value of $a_i^n$.

We note that 13 of the 26 groups encountered in \S\ref{section: classification} already appeared in the classification of \cite{FKRS12}, and we do not need to compute their moments again. We proceed to the computation of the moment sequences for the restriction of~$a_i$ to every connected component of each of the remaining groups. Let $t$ (resp.~$s$) denote the trace of a random element in $\Unitary(1)$ (resp. $\SU(2)$). Recall that
\begin{equation}\label{equation: momg1}
M_{2n}[t]=\binom{2n}{n}\,,\qquad M_{2n}[s]=\frac{1}{n+1}\binom{2n}{n}\,,
\end{equation} 
whereas the odd moments are all zero in both cases.

\textbf{The group $D$.} In this case we have a single connected component, whose moments can be computed by noting that
\begin{align*}
M_{n}[a_1(g)\,|\,g\in D] &=\Exp[(-s^3+2s)^n]\,,\\
M_{n}[a_2(g)\,|\,g\in D] &=\Exp[(s^4-3s^2+2)^n]\,,
\end{align*}
and then applying the second equality in (\ref{equation: momg1}).

\textbf{The groups $\Unitary(2)$ and $N(\Unitary(2))$.} We can use the isomorphism $\Unitary(2)\simeq\Unitary(1)\times\SU(2)/\langle-1\rangle$ to deduce that 
\begin{align*}
M_n[a_1(g)\,|\,g\in\Unitary(2)] &= \Exp[(-t·s)^n]\,,\\
M_n[a_2(g)\,|\,g\in\Unitary(2)] &= \Exp[(s^2+t^2-2)^n]\,,
\end{align*}
and, if $J$ is as in \S\ref{section: Unitary(2)}, that
\begin{align*}
M_n[a_1(g)\,|\,g\in J\Unitary(2)] &= 0\,,\\
M_n[a_2(g)\,|\,g\in J\Unitary(2)] &= \Exp[(-s^2+2)^n]\,.
\end{align*}

\textbf{The groups $C_n$ and $J(C_n)$.} We have $a_1(g)=0$ and $a_2(g)=2$ for any element $g$ in the connected component of $\zeta_m^kJ$ (where $\zeta_m$ and $J$ are as in \S\ref{section: unitary}). Let $C(\zeta_m^k)$ denote the connected component of the matrix $\zeta_m^k$. Then

\begin{align*}
M_{n}[a_1(g)\,|\,g\in C(\zeta_m^k)] &=\frac{2^{n-1}}{\pi}\int_0^{2\pi} \left(\cos\left(3r+\frac{2\pi k}{m}\right)+\cos(r)\right)^ndr\,,\\
M_{n}[a_2(g)\,|\,g\in C(\zeta_m^k)] &=\frac{2^{n-1}}{\pi}\int_0^{2\pi} \left(1+\cos\left(4r+\frac{2\pi k}{m}\right)+\cos\left(2r+\frac{2\pi k}{m}\right)\right)^ndr\,.
\end{align*}

%The odd moments are zero. On the table we have written
%\begin{equation}\label{equation:auxform}
%f(t)=\sqrt{4-(t^3-3t)^2},\qquad
%\end{equation}
%\begin{table}
%\begin{center}
%\setlength{\extrarowheight}{0.5pt}
%\caption{Characteristic polynomial coefficients for a random element on the connected %component of the matrix $\zeta_n$ in terms of the trace of a random element on %$\Unitary(1)$. Here $f(t)$, are as in \ref{equation:auxform}.} \label{table:randomcoeffs}
%\vspace{12pt}
%\begin{tabular}{rrr}
%$n$ & $p_1(t)$  & $p_2(t)$  \\\hline\vspace{-10pt}\\
%$1$ & $t^3-2t$  &  $(t^2-2)(t^2-1)$       \\
%$2$ & $4t-t^3$  &  $-(t^2-2)^2-t^2+6$ \\
%$3$ & $\frac{-t^3+5t}{2}-\frac{\sqrt 3}{2}f(t)$  &         \\
%$4$ & $t-f(t)$  &         \\
%$6$ & $\frac{t^3-t}{2}-\frac{\sqrt 3}{2}f(t)$  &         \\
%\hline
%\end{tabular}
%\end{center}
%\end{table}

\begin{table}
\begin{center}
\setlength{\extrarowheight}{0.5pt}
\caption{Moments $M_n=M_n[a_1]$ for the groups listed in Table~\ref{table:STgroups}.}\label{table:a1moments}
\vspace{8pt}
\begin{tabular}{lrrrrrrrrrrr}
$G$ &  $M_2$ & $M_4$ & $M_6$ & $M_8$ & $M_{10}$ & $M_{12}$ & $M_{14}$ & $M_{16}$ \\\hline\vspace{-8pt}\\
$C_1$ & 4 & 44 & 580 & 8092 & 116304 & 1703636 & 25288120 & 379061020\\
$C_2$ & 4 & 36 & 400 & 4956 & 65904 & 919116 & 13236080 & 194789660\\
$C_3$ & 4 & 36 & 400 & 4900 & 63504 & 854216 & 11806652 & 166685220\\
$C_4$ & 4 & 36 & 400 & 4900 & 63504 & 853776 & 11778624 & 165640540\\
$C_6$ & 4 & 36 & 400 & 4900 & 63504 & 853776 & 11778624 & 165636900\\
$J(C_1)$ & 2 & 22 & 290 & 4046 & 58152 & 851818 & 12644060 & 189530510\\
$J(C_2)$ & 2 & 18 & 200 & 2478 & 32952 & 459558 & 6618040 & 97394830\\
$J(C_3)$ & 2 & 18 & 200 & 2450 & 31752 & 427108 & 5903326 & 83342610\\
$J(C_4)$ & 2 & 18 & 200 & 2450 & 31752 & 426888 & 5889312 & 82820270 \\
$J(C_6)$ & 2 & 18 & 200 & 2450 & 31752 & 426888 & 5889312 & 82818450  \\
$D$ & 1 & 4 & 34 & 364 & 4269 & 52844 & 679172 & 8976188 \\ 
$\Unitary(2)$  & 2 & 12 & 100 & 980 & 10584 & 121968 & 1472328 &
18404100  \\
$N(\Unitary(2))$ & 1 & 6 & 50 & 490 & 5292 & 60984 & 736164 & 9202050\\
$F$ & 4 & 36 & 400 & 4900 & 63504 & 853776 & 11778624 & 165636900\\
$F_a$ & 3 & 21 & 210 & 2485 & 31878 & 427350 & 5891028 & 82824885\\
$F_c$ & 2 & 18 & 200 & 2450 & 31752 & 426888 & 5889312 & 82818450\\
$F_{ab}$ & 2 & 18 & 200 & 2450 & 31752 & 426888 & 5889312 & 82818450\\
$F_{ac}$ & 1 & 9 & 100 & 1225 & 15876 & 213444 & 2944656 & 41409225\\
$F_{a,b}$ & 2 & 12 & 110 & 1260 & 16002 & 213906 & 2946372 & 41415660\\
$F_{ab,c}$ & 1 & 9 & 100 & 1225 & 15876 & 213444 & 2944656 & 41409225\\
$F_{a,b,c}$ & 1 & 6 & 55 & 630 & 8001 & 106953 & 1473186 & 20707830\\
$G_{1,3}$ & 3 & 20 & 175 & 1764 & 19404 & 226512 & 2760615 & 34763300\\
$N(G_{1,3})$ & 2 & 11 & 90 & 889 & 9723 & 113322 & 1380522 & 17382365\\
$G_{3,3}$ & 2 & 10 & 70 & 588 & 5544 & 56628 & 613470 & 6952660\\
$N(G_{3,3})$ & 1 & 5 & 35 & 294 & 2772 & 28314 & 306735 & 3476330\\
$\USp(4)$  & 1 & 3 & 14 & 84 & 594 & 4719 & 40898 & 379236\\\hline
\end{tabular}
\end{center}
\end{table}

\begin{table}
\begin{center}
\setlength{\extrarowheight}{0.5pt}
\caption{Moments of $M_n=M_n[a_2]$ for the groups listed in Table~\ref{table:STgroups}.}\label{table:a2moments}
\vspace{12pt}
\begin{tabular}{lrrrrrrrrrrr}
$G$ &  $M_1$ & $M_2$ & $M_3$ & $M_4$ & $M_{5}$ & $M_{6}$ & $M_{7}$ & $M_{8}$ & $M_{9}$\\\hline\vspace{-8pt}\\
$C_1$ & 2 & 8 & 38 & 196 & 1052 & 5774 & 32146 & 180772 & 1024256\\%, 5837908,
%33433996, 192239854, 1109049320, 6416509142, 37215072638, 216309089956 ]
$C_2$ & 2 & 8 & 32 & 148 & 712 & 3614 & 18916 & 101700 & 557384\\ %,, 3101428, 17456056,
%99127054, 566801900, 3258533958, 18813973012, 109004049764 ]
$C_3$ & 2 & 8 & 32 & 148 & 712 & 3584 & 18496 & 97444 & 521264\\ %,, 2823808, 15458236,
%85387504, 475369988, 2664814832, 15030002822, 85234383844 ]
$C_4$ & 2 & 8 & 32 & 148 & 712 & 3584 & 18496 & 97444 & 521096\\ %,, 2820448, 15414016,
%84918574, 471007916, 2627705760, 14734037152, 82986057764 ]
$C_6$ & 2 & 8 & 32 & 148 & 712 & 3584 & 18496 & 97444 & 521096\\ %,, 2820448, 15414016,
%84917584, 470982176, 2627289344, 14728751872, 82928400164 ]
$J(C_1)$ & 2 & 6 & 23 & 106 & 542 & 2919 & 16137 & 90514 & 512384\\ %,, 2919466, 16718022,
%96121975, 554528756, 3208262763, 18607552703, 108154577746 ]
$J(C_2)$  & 2 & 6 & 20 & 82 & 372 & 1839 & 9522 & 50978 & 278948\\%, 1551226, 8729052,
%49565575, 283405046, 1629275171, 9407002890, 54502057650 ]
$J(C_3)$ & 2 & 6 & 20 & 82 & 372 & 1824 & 9312 & 48850 & 260888\\%, 1412416, 7730142,
%42695800, 237689090, 1332415608, 7515017795, 42617224690 ]
$J(C_4)$ & 2 & 6 & 20 & 82 & 372 & 1824 & 9312 & 48850 & 260804\\%, 1410736, 7708032,
%42461335, 235508054, 1313861072, 7367034960, 41493061650 ]
$J(C_6)$ & 2 & 6 & 20 & 82 & 372 & 1824 & 9312 & 48850 & 260804\\%, 1410736, 7708032,
%42460840, 235495184, 1313652864, 7364392320, 41464232850 ]
$D$ & 1 & 2 & 5 & 16 & 62 & 272 & 1283 & 6316 & 31952\\% 164768, 862138, 4564210, 24399676, 131522588,\\ 
$\Unitary(2)$ & 1 & 4 & 11 & 44 & 172 & 752 & 3383 & 15892 & 76532\\%, 377168, 1891628, 9630832, 49650928, 258746048, 1361027767, 7217806852 
$N(\Unitary(2))$ & 1 & 3 & 7 & 25 & 91 & 386 & 1709 & 7981 & 38329\\%, 188710, 
$F_\nothing$ & 2 & 8 & 32 & 148 & 712 & 3584 & 18496 & 97444 & 521096\\%2820448, 15414016, 84917584, 470982176, 2627289344,
$F_a$ & 2 & 6 & 20 & 82 & 372 & 1824 & 9312 & 48850 & 260804\\
$F_c$ & 1 & 5 & 16 & 77 & 356 & 1802 & 9248 & 48757 & 260548\\
$F_{ab}$ & 2 & 6 & 20 & 82 & 372 & 1824 & 9312 & 48850 & 260804\\%1410736, 7708032, 42460840, 
$F_{ac}$ & 1 & 3 & 10 & 41 & 186 & 912 & 4656 & 24425 & 130402\\
$F_{a,b}$ & 2 & 5 & 14 & 49 & 202 & 944 & 4720 & 24553 & 130658\\
$F_{ab,c}$ & 1 & 4 & 10 & 44 & 186 & 922 & 4656 & 24460 & 130402\\
$F_{a,b,c}$ & 1 & 3 & 7 & 26 & 101 & 477 & 2360 & 12294 & 65329\\
$G_{1,3}$ & 2 & 6 & 20 & 76 & 312 & 1364 & 6232 & 29460 & 142952\\
$N(G_{1,3})$ & 2 & 5 & 14 & 46 & 172 & 714 & 3180 & 14858 & 71732\\
$G_{3,3}$ & 2 & 5 & 14 & 44 & 152 & 569 & 2270 & 9524 & 41576\\
$N(G_{3,3})$ & 1 & 3 & 7 & 23 & 76 & 287 & 1135 & 4769 & 20788\\
$\USp(4)$  & 1 & 2 & 4 & 10 & 27 & 82 & 268 & 940 & 3476\\\hline %,, 13448,
\end{tabular}
\end{center}
\end{table}

\section{Modular forms and Hecke characters}\label{section: modular forms}

Modular forms and Hecke characters play a key role in many of our motive constructions.
Before giving explicit examples, we first recall some theoretical facts concerning modular forms with complex multiplication (CM), following the exposition given in \cite[Chap. II]{Sc06}. These facts allow us to actually prove equidistribution in several cases (see Lemma~\ref{lemma:corrmf}), and they facilitate our numerical computations (via Lemma~\ref{lemma:wt4cm}).
\bigskip

\noindent
\textbf{Notation}:
To avoid potential confusion with the normalized $L$-polynomial coefficients $a_1$ and $a_2$ (and the integer $L$-polynomial coefficients $c_1$ and $c_2$), we generally use $b_n$ (or~$d_n$ or $e_n$) to denote the Fourier coefficients of a modular form $f=f(z)=\sum b_nq^n$, where $q=\exp(2\pi i z)$.  Unless otherwise indicated, the symbols $\omega$ and $i$ denote, respectively, the third and fourth roots of unity in the upper half plane. 

When possible, we identify specific modular forms by their labels in the LMFDB database of $L$-functions, modular forms, and related objects \cite{LMFDB}. These identifiers are formatted as $N.k.cs$, where $N$ is the level, $k$ is the weight, $c$ is an index indicating the character, and $s$ is an alphabetic string that distinguishes the form from others of the same weight, level, and character. The trivial character is always indexed by the label $c=1$.

\subsection{Newforms with complex multiplication}

Let $S_k(\Gamma_1(N))$ denote the complex space of weight $k$ cusp forms for $\Gamma_1(N)$. There is a decomposition 
$$
S_k(\Gamma_1(N))=\bigoplus_\varepsilon S_k(\Gamma_0(N),\varepsilon),
$$
where $\varepsilon\colon (\Z/N\Z)^*\rightarrow \C^*$ runs over the characters of $(\Z/N\Z)^*$ and $S_k(\Gamma_0(N),\varepsilon)$ denotes the space of weight $k$ cusp forms for $\Gamma_0(N)$ with nebentypus $\varepsilon$. We denote by $S_k^{\mathrm{new}}(\Gamma_1(N))$ the complex subspace generated by the newforms. We say that $f=\sum_{n\geq 1}b_n q^n\in S_k(\Gamma_1(N))$ is a \emph{newform} if it is an eigenform for all the Hecke operators, it is new at level $N$, and it is normalized so that $b_1=1$.

The newform $f\in S_k(\Gamma_1(N))$ is said to have complex multiplication (CM) by a (quadratic) Dirichlet character $\chi$ if $b_p=\chi(p)b_p$ for a set of primes of density~$1$.
 
Let $K$ be a quadratic imaginary field, $\gM$ an ideal of $K$, and $l\in \N$. Let $I_{\gM}$ stand for the group of fractional ideals of $K$ coprime to $\gM$. A Hecke character of $K$ of modulus $\gM$ and infinite type $(l,0)$, or simply $l$, is a homomorphism
$$
\psi\colon I_{\gM}\rightarrow \C^*
$$
such that $\psi(\alpha\mathcal O_K)=\alpha^l$ for all $\alpha\in K^ *$ with\footnote{To simplify notation, we will simply write $\equiv$, but the reader should be aware that in this context we are alluding to multiplicative congruence by this sign.} $\alpha\equiv 1 \pmod{\gM}$.
We extend~$\psi$ by defining it to be $0$ for all fractional ideals of $K$ that are not coprime to $\gM$. We say that $\gM$ is the conductor of $\psi$ if the following holds: if $\psi$ is defined modulo $\gM'$, then $\gM|\gM'$. The $L$-function of $\psi$ is then defined by
$$
L(\psi,s):=\prod_{\fp}(1-\psi(\fp)N(\fp)^{-s})^{-1}\,,
$$
where the product runs over all prime ideals of $K$. Let $\Delta_K$ denote the absolute value of the discriminant of $K$ and let $\chi_K$ denote the Dirichlet character associated to $K$. By results of Hecke and Shimura, the inverse Mellin transform 
$$
f_{\psi}:=\sum_{\mathfrak a\subseteq \mathcal O_K}\psi(\mathfrak a)q^ {N(\mathfrak a)}=:\sum_{n\geq 1} b_nq^n
$$
of $L(\psi,s)$ is an eigenform of weight $l+1$, level $\Delta_KN(\gM)$, and nebentypus $\chi_K\eta$, where
$$
\eta(n)=\frac{\psi(n\mathcal O_K)}{n^ l}\quad\text{ if }(n,N(\gM))=1\,,
$$
and $\eta(n)=0$, otherwise.
Moreover, $f_\psi$ is new at this level if and only if $\gM$ is the conductor of $\psi$ and, by construction, we have $b_n=\chi_K(n)b_n$. Thus the modular form $f_\psi$ has CM by $\chi_K$ (we also say that $f_\psi$ has CM by $K$).
It follows from results of Ribet that every CM newform in $S_k(\Gamma_1(N))$ arises in this way; see Proposition~4.4 and Theorem 4.5 in \cite{Ri77}.

In this article we only consider newforms with rational coefficients. The following result describes the nebentypus in this case.
\begin{proposition}[\cite{Sc06}, Cor. II.1.2]\label{proposition: nebentypus}
Let $f\in S_k(\Gamma_1(N))$ be a newform with real coefficients.
\begin{enumerate}[i)]
\item If $k$ is even then the nebentypus $\varepsilon$ is trivial. 
\item If $k$ is odd then the nebentypus $\varepsilon$ is quadratic and $f$ has CM by $\varepsilon$.
\end{enumerate}
\end{proposition}

To ease notation, when the nebentypus is trivial, we simply write $S_k(N)$ in place of $S_k(\Gamma_0(N),\varepsilon_{\rm{triv}})$ and we use $S_k^{\mathrm{new}}(N)$ to denote the subspace of $S_k(N)$ generated by newforms.

We now describe two constructions that play a key role in what follows.
These involve certain weight 4 newforms with CM by $K=\Q(i)$ or $K=\Q(\omega)$, and twists of these forms by a quartic or sextic character (respectively).  We first recall two definitions.

Let $K=\Q(i)$. The biquadratic residue symbol of $\alpha\in\mathcal O_K=\Z[i]$ is the homomorphism
$$
\rs{\alpha}{\cdot}_4\colon I_{((1+i)\alpha)}\rightarrow \mathcal O_K^ *=\langle i\rangle
$$
uniquely characterized by the property that
$$
\alpha^{(N(\fp)-1)/4}\equiv \rs{\alpha}{\fp}_4\, (\bmod \,\fp)\,.
$$
Using biquadratic reciprocity, one can show that this is a Hecke character of infinite type $0$.
We define $\rs{\alpha}{\cdot}_4$ to be zero at fractional ideals of $K$ that are not coprime to $(i+1)\alpha$.

Now let $K=\Q(\omega)$. The sextic residue symbol of $\alpha\in\mathcal O_K=\Z\oplus \omega \Z$ is the homomorphism
$$
\rs{\alpha}{\cdot}_6\colon I_{(2\sqrt{-3}\alpha)}\rightarrow \mathcal O_K^ *=\langle \omega\rangle
$$
uniquely characterized by the property that
$$
\alpha^{(N(\fp)-1)/6}\equiv \rs{\alpha}{\fp}_6\, (\bmod \,\fp)\,.
$$
Using cubic reciprocity, one can show that it is also a Hecke character of infinite type $0$.
We define $\rs{\alpha}{\cdot}_6$ to be zero at fractional ideals of~$K$ that are not coprime to $2\sqrt{-3}\alpha$.

\subsection{CM newforms of weights 3 and 4 with a quartic twist}\label{subsection: psi-4}
Let $K=\Q(i)$. For any prime ideal $\fp$ of $K$ there exists $\alpha_\fp\in \mathcal O_K$ such that $\fp=(\alpha_\fp)$, and if $\fp$ is coprime to $1+i$, then by multiplying $\alpha_\fp$ by an element of $\mathcal O_K^*=\langle i \rangle$, we may assume that $\alpha_\fp\equiv 1\bmod (1+i)^3$. Moreover, this uniquely determines $\alpha_\fp$ (see \cite[Chap. 9, Lemma 7]{IR82}). Now define
$$
\psi(\fp):=\alpha_\fp\,.
$$
This is a Hecke character of infinite type $1$ and conductor $\gM=(1+i)^ 3$. 
By \cite[Chap. 18, \S4]{IR82}, this is the Hecke character attached to the elliptic curve $y^2=x^3-x$. The newform $f_\psi\in S_2^{\mathrm{new}}(32)$ has rational coefficients and LMFDB identifier \href{http://www.lmfdb.org/ModularForm/GL2/Q/holomorphic/32/2/1/a/}{32.2.1a}.

The Hecke character $\psi^3$ has infinite type $3$ and conductor $\gM=(1+i)^3$.
Thus $f_{\psi^3}$ is a newform in $S_4^{\mathrm{new}}(32)$, and its identifier is \href{http://www.lmfdb.org/ModularForm/GL2/Q/holomorphic/32/4/1/b/}{32.4.1b}.

Let $\phi:=\rs{3}{\cdot}_4$. The Hecke character $\psi^3\otimes \phi$ has infinite type $3$, but we do not necessarily know its conductor \emph{a priori}.
However, we may use the above recipe to compute $\psi$ and the first several Fourier coefficients of $f_{\psi^3\otimes \phi}=\sum_{n\geq 1}b_n q^n$; for primes $p\equiv 1\,\bmod 4$ with $3<p\leq 97$, we obtain
$$
\begin{array}{rrrrrrrrrrrr}\hline
p\colon & 5 & 13 & 17 & 29 & 37 & 41 & 53 & 61 & 73 & 89 & 97\\
b_p\colon & 4 & -18 & - 104 & 284 & -214 & -472 & 572 & -830 & -1098 & 176 & -594\\\hline
\end{array}
$$
Let $\chi\colon (\Z/24\Z)^*\rightarrow \C^*$ denote the quadratic Dirichlet character defined by
$$
\chi(n):=
\begin{cases}
1 & \text{if $n\equiv\, 1,\,7,\,17,\,23 \bmod 24$;}\\
-1 & \text{if $n\equiv\, 5,\,11,\,13,\,19 \bmod 24$.}
\end{cases}
$$ 
One may verify that that the Fourier coefficients of $f_{\psi^3\otimes \phi}\otimes\chi$ coincide with those of a newform of weight 4 and level 288.
Moreover, we have
$$
f_{\psi^3\otimes\rs{3}{\cdot}_4}\otimes \chi=f_{\psi^3\otimes\rs{-3}{\cdot}_4},
$$
thus it is a quartic twist of
\href{http://www.lmfdb.org/ModularForm/GL2/Q/holomorphic/32/4/1/b/}{32.4.1b}.\footnote{If $f\in S_k^{\mathrm{new}}(N)$ is an eigenform and $\chi:(\Z/M\Z)^*\rightarrow \C^*$ is a Dirichlet character, then $f\otimes\chi$ is a (not necessarily new) eigenform of $S_k(\lcm(N,M^2))$. The minimal level of $f_{\psi^3\otimes \phi}$ should thus be a divisor of $576$.
Data for this level is not yet available in \cite{LMFDB}, but one may use \cite{Magma} or  \cite{Sage} to identify $f_{\psi^3\otimes \phi}=f\otimes \chi$ as a newform at level 576.}

The Hecke character $\psi^2$ has infinite type $2$ and conductor $(1+i)^2$. Indeed, observe that for $\alpha\in K^ *$ we have
$$
\psi^2(\alpha\mathcal O_K)=\psi(\alpha^2\mathcal O_K)
$$
and $\alpha^2\equiv 1 \bmod (1+i)^3$ if $\alpha\equiv 1\bmod (1+i)^2$.
Thus $f_{\psi^2}$ is a newform in $S_3^{\mathrm{new}}(\Gamma_1(16))$. Let $\phi:=\rs{27}{\cdot}_4$. Proceeding as in the previous case, one may show that $f_{\psi^2\otimes \phi}=\sum_{n\geq 1}b_n q^n$ is new at level $576$ and that its first Fourier coefficients, for primes $p\equiv 1\,\bmod 4$ with $3<p\leq 97$, are
$$
\begin{array}{rrrrrrrrrrrr}\hline
p\colon & 5 & 13 & 17 & 29 & 37 & 41 & 53 & 61 & 73 & 89 & 97\\
b_p\colon & -8 & -10 & 16 & 40 & -70 & -80 & -56 & -22 & 110 & 160 & -130\\\hline
\end{array}
$$

\subsection{A weight 4 CM newform with cubic and sextic twists}\label{subsection: psi-3}
Let $K=\Q(\omega)$.
Since $K$ has class number $1$, for any prime ideal $\fp$ of $K$ there exists $\alpha_\fp\in \mathcal O_K$ such that $\fp=(\alpha_\fp)$.
For $\fp$ coprime to $2\sqrt{-3}$, by multiplying~$\alpha_\fp$ by an element of $\mathcal O_K^*=\langle \omega \rangle$, we may assume that $\alpha_\fp\equiv 1 \bmod 3$, and this uniquely determines $\alpha_\fp$ (see \cite[Prop. 9.3.5]{IR82}).
We now define
$$
\psi(\fp):=\alpha_\fp\,.
$$
This is the Hecke character of infinite type $1$ and conductor $\gM=(3)$ attached to the elliptic curve $y^2+y=x^3$. 
The newform $f_\psi\in S_2^{\mathrm{new}}(27)$ has rational coefficients and identifier \href{http://www.lmfdb.org/ModularForm/GL2/Q/holomorphic/27/2/1/a/}{27.2.1a}.

The Hecke character $\psi^3$ has infinite type $3$ and conductor $\gM=\bigl(\sqrt{-3}\bigr)$. Indeed, observe that for $\alpha\in K^ *$ we have
$$
\psi^3(\alpha\mathcal O_K)=\psi(\alpha^3\mathcal O_K)
$$
and $\alpha^ 3\equiv 1 \bmod 3$ if $\alpha\equiv 1\bmod \sqrt{-3}$. Thus $f_{\psi^3}$ is a newform in $S_4^{\mathrm{new}}(9)$, and its identifier is \href{http://www.lmfdb.org/ModularForm/GL2/Q/holomorphic/9/4/1/a/}{9.4.1a}.

Let $\phi:=\rs{2}{\cdot}_6$. The Hecke character $\psi^3\otimes \phi$ has infinite type $3$.
As before we compute $\psi$ and the first several Fourier coefficients of $f_{\psi^3\otimes \phi}=\sum_{n\geq 1}b_n q^n$; for primes $p\equiv 1\,\bmod 6$ with $3<p\leq 97$, we obtain
$$
\begin{array}{rrrrrrrrrrrr}\hline
p: & 7 & 13 & 19 & 31 & 37 & 43 & 61 & 67 & 73 & 79 & 97 \\
b_p: & 17 & -89 & -107 & 308 & 433 & 520 & 901 & -1007 & -271 & 503 & 1853 \\\hline
\end{array}
$$
Let $\chi\colon (\Z/24\Z)^*\rightarrow \C^*$ denote the quadratic Dirichlet character defined by
$$
\chi(n):=
\begin{cases}
1 & \text{if $n\equiv\, 1,\,5,\,7,\,11\bmod 24$;}\\
-1 & \text{if $n\equiv\, 13,\,17,\,19,\,23\bmod 24$.}
\end{cases}
$$ 
One may verify that the Fourier coefficients of $f_{\Psi\otimes \phi}\otimes\chi$ coincide with those of a newform of weight 4 and level 108
Moreover, we have
$$
f_{\psi^3\otimes \rs{2}{\cdot}_6}\otimes\chi=f_{\psi^3\otimes \rs{2}{\cdot}_6\otimes\rs{2}{N(\cdot)}}=f_{\psi^3\otimes \rs{4}{\cdot}_3}\,.
$$
Thus $f_{\Psi\otimes\phi}\otimes\chi$ (resp. $f_{\psi^3\otimes \phi}$) is a cubic (resp. sextic) twist of \href{http://www.lmfdb.org/ModularForm/GL2/Q/holomorphic/9/4/1/a/}{9.4.1a}.\footnote{Although we will not need it in what follows, we might ask about the minimal level of $f_{\psi^3\otimes \phi}$.
It must be a divisor of $\rm{lcm}(108,24^2)=1728$. This is again out of the range of \cite{LMFDB}, but one may use \cite{Magma} or \cite{Sage} to determine that $f_{\psi^3\otimes\phi}$ is new at level 1728.}

In \S\ref{section: prod mf} we also consider the newform $f_{\psi^2}\in S_3^{\mathrm{new}}(\Gamma_1(27))$.

\subsection{Computing Fourier coefficients of newforms}

One of the key advantages of working with CM newforms $f_{\psi^2}$ or $f_{\psi^3}$ is that we can derive their Fourier coefficients from the corresponding coefficients of the weight 2 CM newform $f_{\psi}$, which we can compute very quickly.

\begin{lemma}\label{lemma:wt4cm}
Let $\psi$ be a Hecke character of an imaginary quadratic field $K$ and suppose that $f_\psi$ has trivial nebentypus.
Suppose that we have Fourier $q$-expansions $f_\psi=\sum b_nq^n$, $f_{\psi^2}=\sum d_nq^n$, and $f_{\psi^3}=\sum e_nq^n$.  Then
\begin{equation}\label{equation: rel coefs}
d_p=b_p^2-2p\qquad \text{and}\qquad e_p = b_p^3-3pb_p
\end{equation}
for primes $p$ that split in $K$. For primes $p$ inert in $K$, we have $d_p=e_p=0$.
\end{lemma}

\begin{proof}
If $p=\fp\overline\fp$ splits in $K$, then $b_p=\psi(\fp)+\psi(\overline\fp)$, and the $p$th Fourier coefficients of $f_{\psi^2}$ and $f_{\psi^3}$ are given by
\begin{align*}
d_p &=\psi(\fp)^2+\psi(\overline\fp)^2= (\psi(\fp)+\psi(\overline\fp))^2 - 2\psi(\fp)\psi(\overline\fp)=b_p^2-2p\,,\\
e_p &=\psi(\fp)^3+\psi(\overline\fp)^3 = (\psi(\fp)+\psi(\overline\fp))^3 - 3\psi(\fp)\psi(\overline\fp)(\psi(\fp)+\psi(\fp)) = b_p^3-3pb_p\,.
\end{align*}
If $p$ is inert in $K$, then $d_p=e_p=0$, because $f_{\psi^ 2}$ and $f_{\psi^3}$ have CM by $K$.
\end{proof}

We note that, in particular, the Fourier coefficients $b_p$ of  \href{http://www.lmfdb.org/ModularForm/GL2/Q/holomorphic/27/2/1/a/}{27.2.1a} (resp.  \href{http://www.lmfdb.org/ModularForm/GL2/Q/holomorphic/32/2/1/a/}{32.2.1a}) and the Fourier coefficients $d_p$ of \href{http://www.lmfdb.org/ModularForm/GL2/Q/holomorphic/9/4/1/a/}{9.4.1a} (resp. \href{http://www.lmfdb.org/ModularForm/GL2/Q/holomorphic/32/4/1/b/}{32.4.1b}) satisfy (\ref{equation: rel coefs}).

%Let $f:=f_\psi$ be a modular form with rational coefficients and CM by a quadratic imaginary field $K$. Let $\phi$ be a finite Hecke character of $K$. Then $f\otimes\phi:=f_{\psi\otimes\phi}$ has rational coefficients if and only if $\ord(\phi)|2$, or $\ord(\phi)=4$ and $K=\Q(i)$, or $\ord(\phi)=3,6$ and $K=\Q(\sqrt{-3})$.

%\begin{lemma}\label{lemma: coefs twist} Let $\psi$ be a Hecke character of a quadratic imaginary field $K$ of infinite type $l$, and let $\phi$ be a Hecke character of finite order dividing $4$ or $6$. Suppose that $f_\psi$ has trivial nebentypus. If $a_n$ denote the Fourier coefficients of $f_\psi$ and $b_n$ denote the Fourier coefficients of $f_{\psi\otimes\phi}$, then $b_p=0$ if $p$ is inert in $K$; and if $p=\fp\overline\fp$ splits in $K$, then
%$$
%b_p=
%\begin{cases}
%a_p & \text{if } \phi(\fp)=1; \\[4pt]
%-a_p & \text{if } \phi(\fp)=-1; \\[4pt]
%\pm t_p & \text{if } \phi(\fp)=i,-i; \\[4pt]
%\frac{a_p}{2}\pm\frac{\sqrt 3}{2}t_p & \text{if } \phi(\fp)=\omega,\omega^5; \\[4pt]
%-\frac{a_p}{2}\pm\frac{\sqrt 3}{2}t_p & \text{if } \phi(\fp)=\omega^2,\omega^4; \\[4pt]
%\end{cases}
%\qquad
%\text{where } t_p=\sqrt{4p^ l-a_p^2}\,.
%$$
%\end{lemma}

Efficiently computing the Fourier coefficients of a general\footnote{Mark Watkins points out that a few examples can be generated using $\eta$ products, whose Fourier coefficients can be computed efficiently using the power series expansion of $\eta$. For example, the form 5.4a used in
Example~\ref{example:mfsum_G33} can be realized as
$\eta(q)^4 \eta(q^5)^4$.}  weight~4 newform is more difficult.
Here we use the modular symbols approach implemented in \cite{Magma} and \cite{Sage}, with a running time of $\tilde{O}(N^2)$.
To improve the constant factors in the running time, we use some specialized C code to handle the most computationally intensive steps, a strategy suggested to us by William Stein.
This optimization speeds up the computation by more than a factor of 100, making it easy to handle norm bounds as large as  $B=2^{24}$.

\section{Direct sum constructions}\label{section: sum}

Following a suggestion of Serre, in this section we construct $M=M_1\oplus M_2$ as the direct sum of 
the Tate twist $M_1$ of the motive associated to a weight~$2$ newform~$f_1$ (with Hodge numbers $h^{2,1} = h^{1,2} = 1$) and the motive $M_2$ associated to a
weight~$4$ newform $f_2$ (with Hodge numbers $h^{3,0} = h^{0,3} = 1$).
We require both~$f_1$ and $f_2$ to have rational Fourier coefficients.

The motive $M$ is defined over $\Q$, but we may also consider its base change to a number field $K$.  Let $f_1=\sum b_nq^n$ and $f_2=\sum d_nq^n$ be the $q$-expansions of~$f_1$ and~$f_2$.  Since~$f_1$ and $f_2$ both have trivial nebentypus (by Proposition~\ref{proposition: nebentypus}), the coefficients of the $L$-polynomial $L_{\fp}(T)$ of the motive $M$ at a prime~$\fp$ of $K$ are given by
\begin{equation}\label{eq:a1a2sum}
c_1 = -(pb_p+d_p)\qquad\text{and}\qquad c_2 = b_pd_p+2p^2,
\end{equation}
where $p=N(\fp)$ and the integer coefficients $c_1$ and $c_2$ are as defined in \eqref{eq:Lpoly}.  The normalized  coefficients are then $a_1(\fp)=c_1/p^{3/2}$ and $a_2(\fp)=c_2/p^2$.

\subsection{Direct sums of uncorrelated newforms}
\label{subsection: Uncorrelated}

We first consider the case where $f_1$ and $f_2$ have no special relationship; the case where $f_1$ and $f_2$ are related (for example, by having the same CM field) is addressed in the next section.
When $f_1$ and $f_2$ are unrelated, we expect the identity component of the Sato-Tate group of $M$ to be one of the three product groups $\Unitary(1)\times\Unitary(1)$, $\Unitary(1)\times \SU(2)$, or $\SU(2)\times \SU(2)$,  depending on whether both, one, or neither of the newforms has complex multiplication.

In the case where exactly one of the forms has complex multiplication, we expect to see the same distribution regardless of which form has CM, and this is confirmed by our numerical experiments.
Thus to facilitate the computations, in both of the first two cases we take $f_2$ to be a CM newform to which Lemma~\ref{lemma:wt4cm} applies, allowing us to use norm bounds $B=2^n$ ranging from $2^{12}$ to $2^{40}$.
In the third case we cannot choose $f_2$ to have CM, which makes the computations more difficult; here we only let $B$ range from $2^{12}$ to~$2^{24}$.
Fortunately there are only two possible Sato-Tate groups in this case and their moments are easily distinguished.

\begin{example}[$\boldsymbol{F, F_a, F_{a,b}}$]\label{example:mfsum_F}
Let $f_1$ be the weight 2 newform \href{http://www.lmfdb.org/ModularForm/GL2/Q/holomorphic/32/2/1/a/}{32.2.1a}, corresponding to (the isogeny class of) the elliptic curve $y^2=x^3-x$, which has CM by $\Q(i)$, and let $f_2$ by the weight 4 newform \href{http://www.lmfdb.org/ModularForm/GL2/Q/holomorphic/9/4/1/a/}{9.4.1a}, which has CM by $\Q(\omega)$.
Moment statistics for the motive $M=M_1\oplus M_2$ over the fields $K=\Q(i,\omega), \Q(\omega),\Q$ are listed in Table~\ref{table:mfsum_F}, along with the corresponding moments for the groups $G=F,F_a,F_{a,b}$.
With $K=\Q(i)$ one obtains essentially the same moment statistics as with $K=\Q(\omega)$; this is as expected, since the groups $F_a$ and $F_b$ are conjugate.
\end{example}

\begin{example}[$\boldsymbol{G_{1,3},N(G_{1,3})}$]\label{example:mfsum_G13}
Let $f_1$ be the weight 2 newform \href{http://www.lmfdb.org/ModularForm/GL2/Q/holomorphic/11/2/1/a/}{11.2.1a}, corresponding to the elliptic curve $y^2+y=x^3-x^2-10x-20$, which does not have CM, and let  $f_2$ by the weight 4 newform \href{http://www.lmfdb.org/ModularForm/GL2/Q/holomorphic/9/4/1/a/}{9.4.1a}, which has CM by $\Q(\omega)$.
Moment statistics for the motive $M=M_1\oplus M_2$ over $K=\Q(\omega),\Q$ are listed in Table~\ref{table:mfsum_F}, along with the corresponding moments for $G=G_{1,3},N(G_{1,3})$.
\end{example}

\begin{example}[$\boldsymbol{G_{3,3}}$]\label{example:mfsum_G33}
Let $f_1$ be the weight 2 newform \href{http://www.lmfdb.org/ModularForm/GL2/Q/holomorphic/11/2/1/a/}{11.2.1a}, and let $f_2$ be the weight 4 newform \href{http://www.lmfdb.org/ModularForm/GL2/Q/holomorphic/5/4/1/a/}{5.4.1a}, neither of which has complex multiplication.
Moment statistics for the motive $M= M_1\oplus M_2$ over $K=\Q$  are listed in Table~\ref{table:mfsum_G33}, along with the corresponding moments for $G=G_{3,3}$.
\end{example}

\subsection{Direct sums of correlated newforms}
\label{subsection: Correlated}

We now consider motives $M=M_1\oplus M_2$ where $M_1$ and $M_2$ are associated to newforms $f_1$ and $f_2$ that bear a special relationship.  More specifically, we shall take $f_1$ to be a weight 2 newform $f_\psi$ with CM by $K$, where $\psi$ is a Hecke character of $K$ (of infinite type~$1$), and then take $f_2$ to be a weight 4 newform $f_{\psi^3\otimes \phi}$, where~$\phi$ is a finite order Hecke character (of infinite type~$0$).
Using variations of the two constructions given in \S\ref{subsection: psi-4} and \S\ref{subsection: psi-3} we are able to construct motives whose $L$-polynomial distributions match all ten of the candidate Sato-Tate groups $G=C_n,J(C_n)$ with identity component $\Unitary(1)$, where $n=1,2,3,4,6$. Moreover, via arguments analogous to those used in \cite{FS12}, we are able to prove equidistribution in each of these cases (alternatively, as we are concerned with a CM construction, equidistribution could be directly deduced from the work of Johansson \cite{Joh14}).

\begin{lemma}\label{lemma:corrmf}
Let $\psi$ be a Hecke character of $K$ of infinite type $1$ such that $f_\psi$ has rational coefficients.  Let $M_1$ be the Tate twist of the motive associated to the weight $2$ newform $f_\psi$. Let $M_2$ be the motive associated to the weight $4$ newform $f_{\psi^ 3\otimes\phi}$, where $\phi$ is a finite order Hecke character (of infinite type~$0$) such that $f_{\psi^ 3\otimes\phi}$ has rational coefficients. The distribution of the normalized Frobenius eigenvalues of $M_1\oplus M_2$ (resp. the extension of scalars $(M_1\oplus M_2)_K$) coincides with the distribution of the eigenvalues of a random element in the group $J(C_{\ord(\phi)})$ (resp. $C_{\ord(\phi)}$).
\end{lemma}

\begin{proof}
Since $f_\psi$ has rational coefficients, its nebentypus is trivial. Thus, if~$p$ is inert in $K$, then the normalized Frobenius eigenvalues of $M_1\oplus M_2$ are $i,-i,i,-i$. It is straightforward to check that, for any $n\in\N$, these are precisely the eigenvalues of the matrix
$$
\begin{pmatrix}
\Theta_n & 0 \\ 0 & \overline \Theta_n
\end{pmatrix} J\,,
$$
where $\Theta_n$ and $J$ are as in \S\ref{section: unitary}. 
If $p=\fp\overline \fp$ splits in $K$, then the Frobenius eigenvalues of $M_1\oplus M_2$ are
$$
N(\fp)\psi(\fp),\, N(\fp)\overline{\psi(\fp)},\,\psi(\fp)^ 3\phi(\fp),\,\overline{\psi(\fp)}^3\overline{\phi(\fp)}\,.
$$
Setting $x_\fp:=\frac{\psi(\fp)}{N(\fp)^{1/2}}$, we find that the normalized Frobenius eigenvalues are
$$
x_\fp,\, \overline x_\fp ,\,\left(x_\fp\right)^ 3\phi(\fp),\,\left(\overline x_\fp\right)^3\overline{\phi(\fp)}\,.
$$

Now let $\gM$ be the conductor of $\phi$, let $K_{\gM}$ be the ray class field of $K$ of modulus $\gM$, and let $(\cdot,K_\gM/K)\colon I_\gM\rightarrow \Gal(K_\gM/K)$ denote the Artin map. Since $(\cdot,K_\gM/K)$ is surjective, for any $\mathfrak a \in I_\gM$ the equality
$$
\tilde\phi((\mathfrak{a},K_\gM/K))=\phi(\mathfrak a)
$$
uniquely determines a character $\tilde\phi\colon \Gal(K_\gM/K)\rightarrow \C^*$. Since the kernel of $\phi$ coincides with the kernel of $(\cdot,K_\gM/K)$ (consisting of those $\alpha\mathcal O_K$ with $\alpha\in K^*$ for which $\alpha\equiv 1 \,(\bmod\, \gM)$), the map $\tilde\phi$ is well defined. We thus have a commutative diagram
$$
\xymatrix{
 I_\gM\ar@{->>}[rd]_{(\cdot,K_\gM/K)} \ar[r]^{\phi} & \C^*\\
 &\Gal(K_\gM/K)\ar[u]^{\tilde \phi}\,,}
$$
with $\tilde\phi$ satisfying $\tilde\phi(\Frob_\fp)=\phi(\fp)$ for every prime $\fp$ coprime to $\gM$.
The lemma then follows from Proposition 3.6 of \cite{FS12}, which asserts that the $x_\fp$'s are equidistributed on $\Unitary(1)$, even when $\fp$ is subject to the condition that $\Frob_\fp=c$ for some fixed conjugacy class $c$ of $\Gal(K_\gM/K)$.
\end{proof}

\begin{example}[$\boldsymbol{C_1,J(C_1)}$]\label{example:mfsum_C1}
Let $f_1=f_\psi$ be the weight 2 newform \href{http://www.lmfdb.org/ModularForm/GL2/Q/holomorphic/27/2/1/a/}{27.2.1a}, corresponding to the elliptic curve $y^2+y=x^3$, and let $f_2=f_{\psi^3}$ be the weight~4 newform 
\href{http://www.lmfdb.org/ModularForm/GL2/Q/holomorphic/9/4/1/a/}{9.4.1a}; both $f_1$ and $f_2$ have CM by $\Q(\omega)$.
Moment statistics for the motive $M = M_1\oplus M_2$ over $K=\Q(\omega),\Q$  are listed in Table~\ref{table:mfsum_C1}, along with the corresponding moments for $G=C_1,J(C_1)$.
\end{example}

\begin{example}[$\boldsymbol{C_2,J(C_2)}$]\label{example:mfsum_C2}
Let $f_1=f_\psi$ be the weight 2 newform \href{http://www.lmfdb.org/ModularForm/GL2/Q/holomorphic/27/2/1/a/}{27.2.1a}, and let $f_2=f_{\psi^3}\otimes \chi_4$ be the weight 4 newform that is the quadratic twist of \href{http://www.lmfdb.org/ModularForm/GL2/Q/holomorphic/9/4/1/a/}{9.4.1a} by the nontrivial Dirichlet character $\chi_4$ of modulus 4; both $f_1$ and $f_2$ have CM by $\Q(\omega)$.
Moment statistics for the motive $M = M_1\oplus M_2$ over $K=\Q(\omega),\Q$  are listed in Table~\ref{table:mfsum_C2}, along with the corresponding moments for $G=C_2,J(C_2)$.
\end{example}

\begin{example}[$\boldsymbol{C_3,J(C_3)}$]\label{example:mfsum_C3}
Let $f_1=f_\psi$ be the weight 2 newform \href{http://www.lmfdb.org/ModularForm/GL2/Q/holomorphic/27/2/1/a/}{27.2.1a}, and let $f_2=f_{\psi^3\otimes\rs{2}{\cdot}_6}\otimes\chi$ be the weight 4 newform that is a cubic twist of \href{http://www.lmfdb.org/ModularForm/GL2/Q/holomorphic/9/4/1/a/}{9.4.1a}, as shown in \S\ref{subsection: psi-3} where $\chi$ is defined; both $f_1$ and $f_2$ have CM by $\Q(\omega)$.
Moment statistics for the motive $M = M_1\oplus M_2$ over $K=\Q(\omega),\Q$  are listed in Table~\ref{table:mfsum_C3}, along with the corresponding moments for $G=C_3,J(C_3)$.
\end{example}

\begin{example}[$\boldsymbol{C_4,J(C_4)}$]\label{example:mfsum_C4}
In this case we may apply a quartic twist to either $f_\psi$ or $f_{\psi^3}$, and it is computationally more convenient to do the former.
So let $f_1$ be the weight 2 newform
corresponding to the elliptic curve $y^2=x^3-2x$, which is a quartic twist of the form $f_\psi=$ \href{http://www.lmfdb.org/ModularForm/GL2/Q/holomorphic/32/2/1/a/}{32.2.1a}.
Let $f_2=f_{\psi^3}$ be the weight~4 newform \href{http://www.lmfdb.org/ModularForm/GL2/Q/holomorphic/32/4/1/b/}{32.4.1b};
both $f_1$ and $f_2$ have CM by $\Q(i)$.
Moment statistics for the motive $M = M_1\oplus M_2$ over $K=\Q(i),\Q$  are listed in Table~\ref{table:mfsum_C4}, along with the corresponding moments for $G=C_4,J(C_4)$.
\end{example}

\begin{example}[$\boldsymbol{C_6,J(C_6)}$]\label{example:mfsum_C6}
Let $f_1=f_\psi$ be the weight 2 newform \href{http://www.lmfdb.org/ModularForm/GL2/Q/holomorphic/27/2/1/a/}{27.2.1a}, and let $f_2=f_{\psi^3\otimes\rs{2}{\cdot}_6}$ be the weight 4 newform of level 576 constructed in \S\ref{subsection: psi-3},
which is a sextic twist of \href{http://www.lmfdb.org/ModularForm/GL2/Q/holomorphic/9/4/1/a/}{9.4.1a}; both $f_1$ and $f_2$ have CM by $\Q(\omega)$.
Moment statistics for the motive $M = M_1\oplus M_2$ over $K=\Q(\omega),\Q$  are listed in Table~\ref{table:mfsum_C6}, along with the corresponding moments for $G=C_6,J(C_6)$.
\end{example}

\section{Tensor product constructions}\label{section: product}

We now consider motives of the form $M=M_1 \otimes M_2$, in which $M_1$ is a 1-motive with Hodge numbers $h^{1,0}=h^{0,1}=1$, and $M_2$ is a 2-motive with Hodge numbers $h^{2,0}=h^{0,2}=1$. We also
consider the related construction in which $M$ is the symmetric cube of $M_1$.

\subsection{Tensor product constructions using elliptic curves}

We first consider examples in which $M_1$ is the 1-motive of an elliptic curve $E_1$ and $M_2$ is the complement of the Tate motive in the symmetric square of an elliptic curve $E_2$ with complex multiplication defined over $K$. 
When $E_1$ does not have complex multiplication, the Sato-Tate group should be $\Unitary(2)$. If $E_1$ has complex multiplication and is not $\overline{K}$-isogenous to $E_2$, the Sato-Tate group should be $F_\nothing$ or $F_c$ depending on whether its complex multiplication is defined over $K$ or not.\footnote{To see why it must be $F_c$, as opposed to $F_a$ or $F_{ab}$, which also have component groups of order 2, note that \eqref{eq:a1a2prod} implies that $G$ must have invariants $z_1=1$ and $z_2=[0,0,0,0,0]$.}
In the case that $E_1$ and $E_2$ are $\overline{K}$-isogenous, the Sato-Tate group should have identity component $\Unitary(1)$; this case is discussed in further detail below.

For any prime $\fp$ of $K$ where both $E_1$ and $E_2$ have good reduction, the coefficients of the $L$-polynomial $L_\fp(T)$ of the motive $M_1\otimes M_2$ can be derived directly from the Frobenius traces $t_1$ and $t_2$ of $E_1$ and $E_2$ at $\fp$.
If $\alpha_1,\overline\alpha_1$ and $\alpha_2,\overline\alpha_2$ are the Frobenius eigenvalues of the reductions of $E_1$ and $E_2$ modulo $\fp$ respectively,  then the Frobenius eigenvalues of $M_1\otimes M_2$ at $\fp$ are $\alpha_1\alpha_2^2,\alpha_1\overline\alpha_2^2,\overline\alpha_1\alpha_2^2$, and $\overline\alpha_1\overline\alpha_2^2$.
The $L$-polynomial coefficients $c_1$ and $c_2$ of \eqref{eq:Lpoly} may be computed via
\begin{equation}\label{eq:a1a2prod}
c_1 = -t_1(t_2^2-2p)\qquad\text{and}\qquad  c_2 = pt_1^2+(t_2^2-2p)^2-2p^2,
\end{equation}
where $p=N(\fp)$, and the normalized coefficients are then $a_1(\fp) = c_1/p^{3/2}$ and $a_2(\fp)=c_2/p^2$.

By using the \texttt{smalljac} software \cite{KS08} to compute the sequences of Frobenius traces of $E_1$ and $E_2$ and applying \eqref{eq:a1a2prod} to the results, we can very efficiently compute the moment statistics of $a_1$ and $a_2$, which allows us to use norm bounds $B=2^n$ ranging from $2^{12}$ to $2^{40}$.

\begin{example}[$\boldsymbol{\Unitary(2)}$]\label{example:ecprod_U2}
Let $E_1$ be the elliptic curve $y^2=x^3+x+1$, which does not have CM, and let $E_2$ be the elliptic curve $y^2=x^3+1$, which has CM by $\Q(\omega)$.
Moment statistics for the motive $M=M_1\otimes M_2$ over $K=\Q(\omega)$ are listed in Table~\ref{table:ecprod_U2}, along with the corresponding moments for the group $G=\Unitary(2)$.
(One can also achieve $N(\Unitary(2))$ by considering this example
over $\Q$; compare Example~\ref{example:mfprod_U2}.)
\end{example}

\begin{example}[$\boldsymbol{F, F_c}$]\label{example:ecprod_F}
Let $E_1$ be the elliptic curve $y^2=x^3+x$ with CM by $\Q(i)$, and let $E_2$ be the elliptic curve $y^2=x^3+1$ with CM by $\Q(\omega)$.
Moment statistics for the motive $M=M_1\otimes M_2$ over $K=\Q(i,\omega),\Q(\omega)$ are listed in Table~\ref{table:ecprod_F}, along with the corresponding moments for $G=F,F_c$.
(One can also achieve $F_{ab}$ and $F_{ab,c}$ by considering this example over $\Q$ and $\Q(\sqrt{3})$; compare Example~\ref{example:mfprod_F}.)
\end{example}

\begin{remark}
Here we appear to be able to realize the Sato-Tate group $F_c$, the first of the three groups ruled out in \cite{FKRS12} for weight~$1$ motives arising from abelian surfaces, and we also appear to realize the second such group, $F_{ab,c}$; see Example \ref{example:mfprod_F}.
It is unclear whether the remaining group $F_{a,b,c}$ ruled out in \cite{FKRS12} can be realized by a weight~$3$ motive with rational coefficients
(but see Example~\ref{example:hecke_char}).
\end{remark}

We now consider the case where $E_1$ and $E_2$ are $\overline{K}$-isogenous.
Without loss of generality (for the purpose of realizing groups) we may suppose that $E_1$ and $E_2$ are actually $\overline{K}$-isomorphic, that is, twists.
The case where $E_1$ and $E_2$ are $K$-isomorphic corresponds to taking the symmetric cube of an elliptic curve, which we consider in the next section;
here we assume that $E_1$ and $E_2$ are twists that are not isomorphic over $K$.

 If $E_1$ and $E_2$ are quadratic twists, the Sato-Tate group of $M_1\otimes M_2$ will be the same as that of $\Sym^3 M_1$; this is evident from \eqref{eq:a1a2prod}, since multiplying either $t_1$ or $t_2$ by $\chi(p)\in\{\pm 1\}$ for some quadratic character $\chi$ will not change any of the $a_1$ and $a_2$ moments, and these moments determine the Sato-Tate group (as can be seen in Tables~\ref{table:a1moments} and~\ref{table:a2moments}).
However, the situation changes if we take a cubic twist.

\begin{example}[$\boldsymbol{C_3}$]\label{example:ecprod_C3}
Consider the elliptic curves $E_1\colon y^2=x^3+4$ and $E_2\colon y^2=x^3+1$, both of which have CM by $K=\Q(\omega)$.
Moment statistics for $M=M_1\otimes M_2$ over $K=\Q(\omega)$ are listed  in Table~\ref{table:ecprod_C3} along with the corresponding moments for $G=C_3$.  Note that the moment $M_{12}[a_1]=854216$ distinguishes $C_3$, and the moment statistics for $M_{12}[a_1]$ obtained by this example are much closer to this value than any of the other $M_{12}[a_1]$ values in Table~\ref{table:a1moments}.
(One can also achieve $J(C_3)$ by considering this example over $\Q$;
compare Example~\ref{example:mfprod_C3}.)
\end{example}

We also get $C_3$ if we use a sextic twist,  for the same reason that using a quadratic twist yields $C_1$.
One might hope that taking $E_1$ to be a quartic twist of $E_2\colon y^2 = x^3 - x$ would yield $C_2$,
but we actually get $C_1$ instead. All of this behavior is explained  by the following lemma
and remark.

\begin{lemma}\label{lemma:CMproddist}
For $A,B\in K^*$, where $K=\Q(\omega)$, let $M_1$ be the $1$-motive of the elliptic curve $E_A\colon y^2=x^3+A$ over $K$, and let $M_2$ be the complement of the Tate motive in the symmetric square of the elliptic curve $E_B\colon y^2=x^3+B$ over~$K$. Let $L=K((A/B)^{1/6})$,  let $d=[L:K]$, and let $n=d/(d,4)$. Then the distribution of the normalized Frobenius eigenvalues of $M_1\otimes M_2$ coincides with the distribution of the eigenvalues of a random element of the group $C_n$.
\end{lemma}

\begin{proof} Let $\End_\Qbar(E_A,E_B)$ denote the ring of endomorphisms from $E_A$ to $E_B$ defined over $\Qbar$. Since $E_A$ and $E_B$ have complex multiplication by $K$ and are isogenous over $L$, the vector space $\End_\Qbar(E_A,E_B)\otimes \Q$ is endowed with the structure of a $K[\Gal(L/K)]$-module; let $\chi$ denote its character. Then for a prime~$\ell$, as in \cite[\S 3.3]{FS12}, we have the following isomorphism of $\overline\Q_\ell[G_K]$-modules 
\begin{equation}\label{equation: decomposition}
V_\ell(E_A)\otimes \overline \Q_\ell\simeq V_\sigma(E_B)\otimes\chi \oplus V_\sigma(E_B)\otimes\overline \chi\,.
\end{equation}
Here $V_\ell(E_A)$ denotes the $\ell$-adic Tate module of $E_A$, $\sigma$ and $\overline \sigma$ stand for the two embeddings of $M$ into $\overline \Q_\ell$, and
$$
V_\sigma(E_B):=V_\ell(E_B)\otimes_{M,\sigma} \overline \Q_\ell\,,\qquad V_{\overline\sigma}(E_B):=V_\ell(E_B)\otimes_{M,\overline\sigma} \overline \Q_\ell\,.
$$
Let $\fp$ be a prime of $K$ of good reduction for $E_A$ and $E_B$ such that $\Frob_\fp$ has order $f$ in $\Gal(L/K)$. Since $\chi$ is injective, it follows from (\ref{equation: decomposition}) that if $\{\alpha_\fp, \overline\alpha_\fp\}$ are the normalized eigenvalues of the action of $\Frob_\fp$ on $V_\ell(E_B)$, then $\{\zeta\alpha_\fp, \overline\zeta\overline\alpha_\fp\}$ are the normalized eigenvalues of the action of $\Frob_\fp$ on $V_\ell(E_A)$, where $\zeta$ is a primitive $f$th root of unity. Thus the normalized eigenvalues of the action of $\Frob_\fp$ relative to $M_1\otimes M_2$ are
\begin{equation}\label{equation: eigenvalues}
\{\zeta\alpha^3_\fp, \overline \zeta\overline\alpha^3_\fp,\overline\zeta\alpha_\fp, \zeta\overline\alpha_\fp\}\,.
\end{equation}
By \cite[Proposition 3.6]{FS12}, the sequence of $\alpha_\fp$'s with $\ord(\Frob_\fp)=f$ is equidistributed on $\Unitary(1)$ with respect to the Haar measure. By the translation invariance of the Haar measure, the sequence of $\beta_\fp$'s with $\ord(\Frob_\fp)=f$ is also equidistributed, where $\beta_\fp:=\overline\zeta\alpha_\fp$. Now \eqref{equation: eigenvalues} reads
$$
\{\zeta^4\beta^3_\fp, \overline \zeta^4\overline\beta^3_\fp,\beta_\fp, \overline\beta_\fp\}\,.
$$
Let $s=f/(f,4)$. We deduce that the normalized eigenvalues of the action of $\Frob_\fp$ relative to $M_1\otimes M_2$ with $\ord(\Frob_\fp)=f$ are equidistributed as the eigenvalues of a random matrix in the connected component of the matrix
$$
\begin{pmatrix} \Theta_s U & 0 \\ 0 & \overline\Theta_s\overline U 
\end{pmatrix}
$$
(in the notation of \S\ref{section: unitary}). The extension $L/K$ is cyclic of order dividing~$6$, which implies that the normalized Frobenius eigenvalues of $M_1\otimes M_2$ have the same distribution as the eigenvalues of a random matrix in the group $C_n$.
\end{proof}

\begin{remark}
The same proof works when $K=\Q(i)$, $L=K((A/B)^{1/4})$, $E_A\colon y^2=x^3+Ax$, and $E_B\colon y^2=x^3+Bx$. In this case, $n=d/(4,d)=1$, and the distribution of the normalized Frobenius eigenvalues of $M_1\otimes M_2$ is thus always governed by $C_1$.
\end{remark}

\subsection{Symmetric cubes of elliptic curves}
\label{subsection: Symmetric cubes}

We next consider motives of the form $M=\Sym^3 M_1$ over a field $K$ in which~$M_1$ is the 1-motive of an elliptic curve $E_1$.
The Sato-Tate group in this case should be $C_1, J(C_1),$ or~$D$, depending on whether $E$ has complex multiplication
defined over $K$, complex multiplication that is not defined over~$K$, or no complex multiplication at all.
This is effectively a special case of the product construction $M_1\otimes M_2$ with $E_1=E_2$, except that we do not necessarily require $E_1=E_2$ to have complex multiplication.
To compute the $L$-polynomial coefficients of $M$ we simply apply the equations in \eqref{eq:a1a2prod} with $t_1=t_2$. 

\begin{example}[$\boldsymbol{C_1,J(C_1),D}$]\label{example:eccube_C1}
See Table~\ref{table:eccube_C1} for moment statistics of the motive $M=\Sym^3 M_1$ in three cases: (1) $E_1$ is the elliptic curve $y^2=x^3+1$ over $\Q(\omega)$; (2) $E_1$ is the elliptic curve $y^2=x^3+1$ over $\Q$; and (3) $E_1$ is the elliptic curve $y^2=x^3+x+1$; along with the corresponding moments for $G=C_1,J(C_1),D$.
\end{example}

\subsection{Tensor product constructions using modular forms}\label{section: prod mf}

We now consider motives $M=M_1\otimes M_2$ that arise as the tensor product of the motive $M_1$
associated to a weight 2 newform $f_1$ (with Hodge numbers $h^{1,0}=h^{0,1}=1$) and the motive $M_2$ associated to a weight 3 newform $f_2$ (with Hodge numbers $h^{2,0}=h^{0,2}=1$).  We assume that both $f_1$ and $f_2$ have rational Fourier coefficients.

By Proposition \ref{proposition: nebentypus}, $f_1$ has trivial nebentypus and $f_2$ has CM by its (quadratic) nebentypus $\chi$. The motive $M$ is defined over $\Q$, and we consider its base change to a number field $K$. If the $q$-expansions of $f_1$ and $f_2$ are given by $f_1=\sum b_nq^n$ and $f_2=\sum d_nq^n$, then the coefficients of the $L$-polynomial $L_\fp(T)$ at a prime $\fp$ of $K$ of good reduction for $M$ are given by
\begin{equation}
c_1 = -b_pd_p\qquad\text{and}\qquad c_2=\chi(p)pb_p^2 + d_p^2-2\chi(p)p^2,
\end{equation}
where $p=N(\fp)$ and the integer coefficients $c_1$ and $c_2$ are as defined in \eqref{eq:Lpoly}.
The normalized  coefficients are then $a_1(\fp)=c_1/p^{3/2}$ and $a_2(\fp)=c_2/p^2$.

\begin{lemma} Let $M_1$ be the motive associated to a weight $2$ newform $f_1$ and let $M_2$ be the motive associated to a weight $3$ newform $f_2$. Assume that both $f_1$ and $f_2$ have rational Fourier coefficients. Then $M=M_1\otimes M_2$ is self-dual.
\end{lemma}

\begin{proof}
It is enough to show that the (normalized) Frobenius eigenvalues of~$M$ at prime of good reduction~$p$ come in conjugate pairs. Let $\alpha_p$ and $\overline\alpha_p$ denote the normalized Frobenius eigenvalues of~$M_1$. We have $\alpha_p\overline\alpha_p=1$, since the nebentypus of $f_1$ is trivial. For the normalized Frobenius eigenvalues of $M_2$ we have two cases according to the value of the nebentypus~$\chi$ of~$f_2$ at~$p$: $(1)$~if $\chi(p)=-1$, then they are~$1$ and~$-1$, since $f_2$ has CM by $\chi$, and $(2)$ if $\chi(p)=1$, then they are $\beta_p$ and $\overline \beta_p$ with $\beta_p\overline \beta_p=1$. In any of the two cases, we readily check that the normalized Frobenius eigenvalues of~$M$ come in conjugate pairs
\begin{align*}
(1): & \quad\{\alpha_p\beta_p,\,\alpha_p\overline\beta_p,\,\overline\alpha_p\beta_p,\,\overline\alpha_p\overline\beta_p\}\,,\\
(2): &  \quad\{\alpha_p,\,-\alpha_p,\,\overline\alpha_p,\,-\overline\alpha_p\}\,.
\end{align*}
Consequently, $M$ is self-dual.
\end{proof}

\begin{remark} With the hypothesis of the lemma, $M_2$ is not self-dual, since the nebentypus of $f_2$ is not trivial (and note therefore that this is not an obstruction for $M_1\otimes M_2$ being self-dual). In particular, seen as a motive over $\Q$, the Sato-Tate group
$$
\left\langle
\begin{pmatrix} 
u & 0 \\ 
0 & \overline u
\end{pmatrix},\,
\begin{pmatrix} 
1 & 0 \\ 
0 & -1
\end{pmatrix},\,
\begin{pmatrix} 
0 & 1 \\ 
-1 & 0
\end{pmatrix}
: u\in\C,\,u\cdot \overline u=1 
\right\rangle
$$
of $M_2$ is a subgroup of $\Unitary(2)$ not contained in $\SU(2)$.
\end{remark}

\begin{example}[$\boldsymbol{C_1,J(C_1)}$]\label{example:mfprod_C1}
Let $f_1=f_\psi$ be the weight 2 newform \href{http://www.lmfdb.org/ModularForm/GL2/Q/holomorphic/27/2/1/a/}{27.2.1a}, corresponding to the elliptic curve $y^2+y=x^3$, and let $f_2=f_{\psi^2}$, which is a weight 3 newform of level 27; both $f_1$ and $f_2$ have CM by $\Q(\omega)$.
Moment statistics for the motive $M = M_1\otimes M_2$ over $K=\Q(\omega),\Q$  are listed in Table~\ref{table:mfprod_C1}, along with the corresponding moments for $G=C_1,J(C_1)$.
\end{example}

\begin{remark}
The sequence of $L$-polynomials of the motive constructed as a tensor product $M_1\otimes M_2$ in Example \ref{example:mfprod_C1}, using $f_1=$ \href{http://www.lmfdb.org/ModularForm/GL2/Q/holomorphic/27/2/1/a/}{27.2.1a} and $f_2=f_{\psi^2}$ is identical to the sequence of $L$-polynomials of the motive constructed as a direct sum $M_1\oplus M_2$ in Example \ref{example:mfsum_C1}, using $f_1=$ \href{http://www.lmfdb.org/ModularForm/GL2/Q/holomorphic/27/2/1/a/}{27.2.1a} and $f_2 =$ \href{http://www.lmfdb.org/ModularForm/GL2/Q/holomorphic/9/4/1/a/}{9.4.1a}.
\end{remark}

\begin{example}[$\boldsymbol{C_2,J(C_2)}$]\label{example:mfprod_C2}
Let $f_1=f_\psi$ be the weight 2 newform \href{http://www.lmfdb.org/ModularForm/GL2/Q/holomorphic/32/2/1/b/}{32.2.1b}, corresponding to the elliptic curve $y^2=x^3-x$, and let $f_2=f_{\psi^2\otimes\phi}$ be the weight~3 newform of level 576 constructed in \S\ref{subsection: psi-4}, which is a quartic twist of $f_{\psi^2}$; both $f_1$ and $f_2$ have CM by $\Q(i)$.
Moment statistics for the motive $M = M_1\otimes M_2$ over $K=\Q(\omega),\Q$  are listed in Table~\ref{table:mfprod_C2}, along with the corresponding moments for $G=C_2,J(C_2)$.
\end{example}

\begin{example}[$\boldsymbol{C_3,J(C_3)}$]\label{example:mfprod_C3}
Let $f_1$ be the weight 2 newform \href{http://www.lmfdb.org/ModularForm/GL2/Q/holomorphic/36/2/1/a/}{36.2.1a}, which is the cubic twist of $f_\psi=$\href{http://www.lmfdb.org/ModularForm/GL2/Q/holomorphic/27/2/1/a/}{27.2.1a} corresponding to the elliptic curve $y^2=x^3+1$, and let $f_2=f_{\psi^2}$, a weight~3 newform of level 27; both $f_1$ and $f_2$ CM by $\Q(\omega)$.
Moment statistics for the motive $M = M_1\otimes M_2$ over $K=\Q(\omega),\Q$  are listed in Table~\ref{table:mfprod_C3}, along with the corresponding moments for $G=C_3,J(C_3)$.
\end{example}

\begin{remark} The behavior observed in the above examples can be explained by means of arguments completely analogous to those of Lemma \ref{lemma:corrmf}.
Let $\psi$ be a Hecke character of a quadratic imaginary field $K$ of infinite type $1$ such that $f_\psi$ has rational coefficients. Let $\phi_1$ (resp.~$\phi_2$) be a Hecke character of finite order~$n$ such that $f_{\psi^2\otimes\phi_1}$ (resp. $f_{\psi\otimes\phi_2}$) has rational coefficients.
We then have:
\begin{enumerate}[(i)] 
\item If $f_1:=f_{\psi^2}$ and $f_2:=f_{\psi\otimes\phi_2}$, then the distribution of the normalized Frobenius eigenvalues of $M_1\otimes M_2$ (resp. of the base change $(M_1\otimes M_2)_K$) coincides with the distribution of the eigenvalues of a random element in $J(C_t)$ (resp. $C_t$), where $t=n/(n,2)$.
\item If $f_1:=f_{\psi^2\otimes\phi_1}$ and $f_2:=f_{\psi}$, then the distribution of the normalized Frobenius eigenvalues of $M_1\otimes M_2$ (resp. of the base change $(M_1\otimes M_2)_K$) coincides with the distribution of the eigenvalues of a random element in $J(C_t)$ (resp. $C_t$), where $t=n/(n,4)$.
\end{enumerate}
\end{remark}

\begin{example}[$\boldsymbol{F,F_{ab},F_c,F_{ab,c}}$]\label{example:mfprod_F}
Let $f_1$ be the weight 2 newform \href{http://www.lmfdb.org/ModularForm/GL2/Q/holomorphic/32/2/1/a/}{32.2.1a}, which has CM by $\Q(i)$, and let $f_2:=f_{\psi^2}$, a weight~3 newform of level 27, which has CM by $\Q(\omega)$.
Moment statistics for the motive $M = M_1\otimes M_2$ over the fields $K=\Q(i,\omega)$, $\Q(\sqrt{3})$, $\Q(i)$, $\Q$  are listed in Table~\ref{table:mfprod_F}, along with the corresponding moments for $G=F,F_{ab},F_c,F_{ab,c}$.
With $K=\Q(\omega)$ one obtains essentially the same moment statistics as with $K=\Q(i)$; this is as expected, since the groups $F_{abc}$ and $F_c$ are conjugate.
\end{example}

\begin{example}[$\boldsymbol{\Unitary(2),N(\Unitary(2))}$]\label{example:mfprod_U2}
Let $f_1$ be the weight 2 newform \href{http://www.lmfdb.org/ModularForm/GL2/Q/holomorphic/11/2/1/a/}{11.2.1a}, corresponding to the elliptic curve $y^2+y=x^3-x^2-10x-20$, which does not have CM, and let $f_2:=f_{\psi^2}$, a weight~3 newform of level 27, which has CM by $\Q(\omega)$.
Moment statistics for the motive $M = M_1\otimes M_2$ over $K=\Q(\omega),\Q$  are listed in Table~\ref{table:mfprod_U2}, along with the corresponding moments for $G=\Unitary(2),N(\Unitary(2))$.
\end{example}

\section{The Dwork pencil}\label{section: Dwork}
We next describe a construction that gives rise to motives whose $L$-polynomial distributions match the group $\USp(4)$, something that cannot be achieved using any of the preceding methods.
To facilitate explicit computations with the Dwork pencil of threefolds,
we work with a family of hypergeometric motives defined by fixed parameters $\alpha=(1/5,2/5,3/5,4/5)$ and $\beta=(0,0,0,0)$, and a varying parameter $z=(5/t)^5$, where $t$ is the Dwork pencil parameter, as described in \cite{CR12}. We first summarize the general setup in \cite{CR12} and then specialize to the case of interest.

\subsection{Trace formulas and algorithms}\label{subsection:trace_formula}

For a prime $p$, let $\Q_{(p)}$ denote the ring of rational numbers with denominators prime to $p$.
For $z \in \Q_{(p)}$, we write $\Teich(z)$ to denote the Teichm\"uller lift of the reduction of $z$ modulo $p$.
Letting $\Gamma_p(x)$ denote the $p$-adic gamma function, for each prime power $q=p^f$ we define $\Gamma_q^*(x):=\prod_{v=0}^{f-1}\Gamma_p(\{p^vx\})$, 
and then define a $p$-adic analogue of the Pochhammer symbol by setting
\[
(x)_{m}^* := \frac{\Gamma_q^*\left(x+\frac{m}{1-q}\right)}{\Gamma_q^*(x)}.
\]
Given vectors $\alpha=(\alpha_1,\ldots,\alpha_r)$ and $\beta=(\beta_1,\ldots,\beta_r)$ in $\Q_{(p)}^r$ and $z\in\Q_{(p)}$, for each prime power $q=p^f$ we define
\begin{equation}\label{eq:Hq}
H_q\left(\begin{matrix}\alpha\\\beta\end{matrix}\Big\vert z\right) := \frac{1}{1-q}\sum_{m=0}^{q-2}(-p)^{\eta_m(\alpha)-\eta_m(\beta)}q^{\xi_m(\beta)}\left(\prod_{j=1}^r \frac{(\alpha_j)_m^*}{(\beta_j)_m^*}\right)\Teich(z)^m,
\end{equation}
using the notations
\[
\eta_m(x_1,\ldots,x_r):=\sum_{j=1}^r\sum_{v=0}^{f-1}\left\{p^v\left(x_j+\frac{m}{1-q}\right)\right\}-\left\{p^vx_j\right\},
\]
and
\[
\xi_m(\beta):=\#\{j:\beta_j=0\}-\#\left\{j:\beta_j+\frac{m}{1-q}=0\right\}.
\]
(with $\beta=(0,0,0,0)$ we have $\xi_m(\beta)=4$ for all nonzero $m$ and $\xi_0(\beta)=0$).

Now let $X_\psi$ be the quintic threefold given in \eqref{eq:quintic},
\[ 
x_0^5 + x_1^5 + x_2^5 + x_3^5 + x_4^5 = t x_0 x_1 x_2 x_3 x_4,
\]
with the parameter $t=5\psi$.
Let $V_\psi$ be the subspace of $H^3(X_\psi,\C)$ fixed by the automorphism group
\[
\{(\zeta_1,\ldots,\zeta_5)|\zeta_i^5=1,\zeta_1\cdots\zeta_5=1\},
\]
acting by $x_i\mapsto\zeta_ix_i$.  For primes $p\ne 5$ for which we have $\psi^5\not\equiv 1\bmod p$ and $\psi\ne\infty\bmod p$, the Euler factor of the $L$-function of $V_\psi$ at $p$ has the form  \eqref{eq:Lpoly},
\[
L_p(T) = p^6T^4+c_1p^3T^3 + c_2pT^2+c_1T+1,
\]
where $c_1$ and $c_2$ are integers.
For $\psi\not\equiv 0\bmod p$, the trace of the geometric Frobenius on $V_\psi$ is given by
\[
\Trace\left(\Frob_q\big\vert_{V_\psi}\right) = H_q\left(\begin{matrix}\frac{1}{5}&\frac{2}{5}&\frac{3}{5}&\frac{4}{5}\\0&0&0&0\end{matrix}\thinspace\Big| \psi^{-5}\right).
\]
Abbreviating the right-hand side as $H_q$, we have
\[
c_1 = -H_p,\qquad\text{and}\qquad c_2=\frac{1}{2p}\left(H_p^2-H_{p^2}\right).
\]

The Weil bounds imply that $|c_1|\le 4p^{3/2}$, so for $p > 64$ it suffices to compute $H_p\bmod p^2$.
Computing $c_2$ requires more precision: we have $-4p^3 < 2pc_2 \le 12p^3$, so for $p>16$ it is enough to compute $H_p$ and $H_{p^2}$ modulo $p^4$.

Specializing $\alpha=(1/5,2/5,3/5,4/5)$ and $\beta=(0,0,0,0)$ in \eqref{eq:Hq} allows us to simplify the formulas.
For the sake of brevity (and ease of computation), we focus on the problem of computing $H_p\bmod p^2$, so $q=p$ and $f=1$.
We have $\eta_0(x)=\xi_0(x)=0$ and $(\alpha_j)_0^*=(\beta_j)_0^*=1$, thus the $m=0$ term in \eqref{eq:Hq} is equal to 1.
For $m>0$ we have $\xi_m(\beta)=4$, and one finds that
\[
\eta_m(\alpha)-\eta_m(\beta)=\begin{cases}
-4\qquad & \text{if } 0 < m < \lfloor\frac{p+4}{5}\rfloor,\\
-3 &\text{if } \lfloor\frac{p+4}{5}\rfloor \le m < \lfloor\frac{2p+3}{5}\rfloor,\\
-2 &\text{if } \lfloor\frac{2p+3}{5}\rfloor \le m < \lfloor\frac{3p+2}{5}\rfloor,\\
-1 &\text{if } \lfloor\frac{3p+2}{5}\rfloor \le m < \lfloor\frac{4p+1}{5}\rfloor,\\
0&\text{if } m \ge \lfloor\frac{4p+1}{5}\rfloor.\\
\end{cases}
\]
When working modulo $p^2$, only the first two ranges of $m$ are relevant (the other terms in the sum are all divisible by $p^2$), and we may write
\[
H_p\left(\begin{smallmatrix}\frac{1}{5}&\frac{2}{4}&\frac{3}{5}&\frac{4}{5}\\0&0&0&0\end{smallmatrix}\Big\vert z\right)\equiv\frac{1+S_1-pS_2}{1-p}\bmod p^2,
\]
where
\[
S_1=\sum_{m=1}^{m_1-1}\left(\prod_{j=1}^4\frac{(j/5)_m^*}{(0)^*_m}\right)\Teich(z)^m,\quad
S_2=\sum_{m=m_1}^{m_2-1}\left(\prod_{j=1}^4 \frac{(j/5)_m^*}{(0)^*_m}\right)\Teich(z)^m,
\]
with $m_1= \lfloor\frac{p+4}{5}\rfloor$ and $m_2=\lfloor\frac{2p+3}{5}\rfloor$.

To compute $H_p\bmod p^2$, it suffices to compute $S_1\bmod p^2$ and $S_2\bmod p$.
Evaluating the Pochhammer symbols $(\cdot)^*_m$ that appear in the formulas for $S_1$ and $S_2$ thus reduces to computing $\Gamma_p(x)$ modulo $p^2$, or modulo $p$.
To compute $\Gamma_p(x)\bmod p^2$ for $x\in \Q_{(p)}$, we first reduce $x$ modulo $p^2$ and use
\begin{equation}\label{eq:gamma_id}
\Gamma_p(x+1) = \begin{cases}
-x\Gamma_p(x)\qquad&\text{for $x\in\Z_p^*$},\\
-\Gamma_p(x)&\text{for $x\in p\Z_p$},
\end{cases}
\end{equation}
to shift the argument down so that it is divisible by~$p$.  We then apply
\[
\Gamma_p(py)\equiv 1+\left(1+\frac{1}{(p-1)!}\right)y\bmod p^2.
\]
For $x=x_0+px_1$ with $0< x_0<p$, we have
\begin{align*}
\Gamma_p(x) &\equiv (-1)^{x_0}(px_1+1)\cdots(px_1+x_0-1)\left(1+\left(1+\frac{1}{(p-1)!}\right)x_1\right)\bmod p^2\\
&\equiv (-1)^{x_0}\left(px_1\sum_{k=1}^{x_0-1}\frac{(x_0-1)!}{k}+(x_0-1)!\right)\left(1+\left(1+\frac{1}{(p-1)!}\right)x_1\right)\bmod p^2.
\end{align*}
To compute $\Gamma_p(x)\bmod p$, simply apply the above formula with $x_1=0$.

Now let $F_n:=n!$ and $T_n:=\sum_{k=1}^n\frac{n!}{k}$.
We may compute $F_n$ and $T_n$ modulo~$p^2$ for $0\le n < p$ via the recurrences $F_{n+1}=(n+1)F_n$ and
$T_{n+1}=(n+1)T_n+F_n$, with $F_0=1$ and $T_0=0$.
Having computed the $F_n$ and $T_n$ using $O(p)$ operations in $\Z/p^2\Z$, we can use the above formulas to compute $\Gamma_p(x)$ for any $x\in\Z/p^2\Z$ using $O(1)$ operations in $\Z/p^2\Z$.
Noting that $\Teich(z)\equiv z^p\bmod p^2$, we can compute $H_p\bmod p^2$ using a total of $O(p)$ operations in $\Z/p^2\Z$.

To efficiently compute the moment statistics of $a_1$ for a large set $S$ of parameter values $z$ in parallel, for each $p$ up to a given bound~$N$ we compute $H_p(z)$ as a polynomial in $\Teich(z)$ with coefficients in $\Z/p^2\Z$.
For $p < \min(\#S,N)$, we then compute $H_p(z^p)\bmod p^2$ for every nonzero $z\in\Z/p\Z$ using fast algorithms for multi-point polynomial evaluation \cite[Alg.~10.8]{GG03}, and construct a lookup table that maps values of $z$ in $\Z/p\Z$ to values of $a$.
If $M = \#S$, then we can compute $H_p(z) \bmod p^2$ for all primes $p\le N$ and all $z\in S$ in time
\[
O\bigl(\pi(N)\Mtime(N)\Mtime(\log N)\log N + M\pi(N)\log N\bigr),
\]
where $\Mtime(n)$ denotes the cost of multiplication.
For $M\ge N$, this corresponds to an average cost of $O((\log N)^{3+o(1)})$ per $H_p(z)$ computation.

Computing the moment statistics of $a_2$ is substantially more work, since we then also need to compute $H_{p^2}(z)$ (modulo $p^4$), which involves $O(p^2)$ arithmetic operations, compared to the $O(p)$ operations needed to compute $H_p(z)$.
To compute $\Gamma_p(x)\bmod p^4$ for $x\in\Q_{(p)}$, we first reduce $x$ modulo $p^4$ and use \eqref{eq:gamma_id} to shift the argument down so that it is divisible by~$p$.  We then apply the formula
\[
\Gamma_p(py) \equiv 1 + a_1y + a_2y^2 + a_3y^3 \bmod p^4,
\]
with
\begin{align*}
a_2 &\equiv -\bigl((p-1)! + 1/(p-1)! + 2\bigr)/2\bmod p^4,\\
a_1 &\equiv -\bigl(8(p-1)! +(2p)!/(2p^2) + 4a_2 + 7\big)/6\bmod p^4,\\
a_3 &\equiv -\bigl((p-1)! + 1 + a_1 + a_2\bigr)\bmod p^4.
\end{align*}
After computing $H_p(z)$ and $H_{p^2}(z)$, one then computes $H_p(z)^2-H_{p^2}\bmod p^4$, lifts this value to an integer that is known to lie in the interval $[-4p^3,12p^3]$, and then divides by $2p$ to obtain the $L$-polynomial coefficient $c_2$, and $a_2=c_2/p^2$.

\begin{remark}
Given the higher cost of computing moment statistics for $a_2$, for the purposes of comparison with $\USp(4)$, we choose to mainly focus on $a_1$. This is reasonable because the $a_1$ moments of $\USp(4)$ easily distinguish it from any of the other candidate Sato-Tate groups, as can be seen in Table~\ref{table:a1moments}.

On the other hand, an ongoing project of the second author with Edgar Costa and David Harvey is expected to yield a computation of $a_2$ using only $O(p)$ arithmetic operations. The strategy is to view the members of the Dwork pencil as nondegenerate toric hypersurfaces, then make a careful computation in $p$-adic cohomology in the style of the work of the second author \cite{Ked01} on hyperelliptic curves.
\end{remark}

Note that the algorithms described above cannot be used when $t=0$,
because then the condition $\psi \not\equiv 0 \pmod{p}$ is never satisfied. For completeness, we describe this case separately.
\begin{example}[$\boldsymbol{F_{ac}}$] \label{example:Fac}
Let $M$ be the motive arising from the quintic threefold~\eqref{eq:quintic}
with parameter $t = 0$. The $L$-polynomials in this case where computed by Weil in terms of Jacobi sums; they coincide with the $L$-polynomials of the unique algebraic Hecke character over $\Q(\zeta_5)$ of conductor
$(1-\zeta_5)^2$ and infinite type $(3,0), (2,1)$.
The latter can be computed efficiently using \cite{Magma}, as demonstrated to us by Mark Watkins.
Moment statistics for the motive $M$ over $K=\Q$ 
 are listed in Table~\ref{table:Fac}, along with the corresponding moments for $G=F_{ac}$.
\end{example}

\subsection{Experimental results}

Using the algorithms described in the previous section, we computed $a_1$ moment statistics for the family of hypergeometric motives with rational parameter~$z$ of height at most $10^3$; the set $S$ of such $z$ has cardinality greater than $10^6$.
We computed $c_1$~values for all $z\in S$ and all $p\le 2^{14}$, and for a subset of the $z\in S$ we continued the computation over $p\le 2^{20}$.
For each value of $z$ we computed the moment statistic $M_n[a_1]$ for $1\le n\le 12$.  In every case the moment statistics appeared to match the $a_1$ moment sequence of $\USp(4)$ listed in Table~\ref{table:a1moments}.
We note that $\USp(4)$ is the only group with $M_4[a_1]=3$, and its sixth moment $M_6[a_1]=14$ is less than half any of the other values for $M_6[a_1]$ listed in Table~\ref{table:a1moments}; these differences are clearly evident in the moment statistics, even when using a norm bound as small as~$B=2^{14}$.

We then conducted similar experiments for each of  the following families:
\begin{itemize}
\item $z=(5/t)^5$ for rational $t$ of height at most 1000;
\item $z=1+1/n$ for integers $n$ of absolute value at most $10^5$.
\item $z = (z_3\zeta^3+z_2\zeta^2+z_1\zeta+z_0)^{-1}$ for a primitive fifth root of unity $\zeta$ and integers $z_0,z_1,z_2$, and $z_3$ of absolute value at most 10.
\end{itemize}
In every case the moment statistics again appeared to match the $\USp(4)$ moment sequence; we found no exceptional cases
aside from the excluded case $t=0$ (see Example~\ref{example:Fac}).

\begin{example}[$\boldsymbol{\USp(4)}$]\label{example:USp4}
Let $M$ be the motive arising from the quintic threefold~\eqref{eq:quintic}
with parameter $t = -5$ (that is, $z=-1$), as described in \S\ref{subsection:trace_formula}, over the field $K=\Q$.
Table~\ref{table:USp4} lists moment statistics of $a_1$ as the norm bound $B=2^n$ varies from $2^{10}$ to $2^{24}$,
and moment statistics of $a_2$ with $B=2^n$ varying from $2^{10}$ to $2^{13}$.
The corresponding moments for the group $G=\USp(4)$ are shown in the last line for comparison.
\end{example}

\begin{remark} \label{R:de jong}
It is worth contrasting the behavior of the Dwork pencil of threefolds with the behavior of a universal family
of elliptic curves, in which one always sees infinitely many curves with complex multiplication.  It has been suggested by de Jong that the scarcity of special members of the Dwork family may be explained by Hodge-theoretic
considerations (unpublished, but see \cite{dJ02}). However, such considerations do not give any indication about the \emph{number} of exceptional cases. It is entirely possible that there are some unobserved
exceptional cases arising at large height and/or over a number field other than~$\Q$.
\end{remark}

\begin{remark}
The Dwork pencil is a family of \emph{hypergeometric motives}, i.e., a family whose Picard-Fuchs equation is a hypergeometric differential equation. One can classify such families for fixed weight and Hodge numbers; for the values we are considering, there are 47 such families
(as verified by the \cite{Magma} command
\texttt{PossibleHypergeometricData}).
The computation of $L$-polynomials in these families has recently been
implemented by Mark Watkins in \cite{Magma}, and leads to some other exceptional cases (e.g., example H126E5 in the \textit{Magma Handbook}).
\end{remark}

\section{More modular constructions}

At this point, all of the groups listed in Table~\ref{table:STgroups}
are accounted for except for $F_{a,b,c}$ and $N(G_{3,3})$. 
We conclude with some more exotic uses of modular forms, leading to a realization of $N(G_{3,3})$ and a tantalizing near-miss for $F_{a,b,c}$. Thanks to Mark Watkins for suggesting these examples
and providing assistance with computations in \cite{Magma}.

\subsection{Hilbert modular forms}

\begin{example}[$\boldsymbol{N(G_{3,3})}$]  \label{example:hmf}
There is a unique normalized Hilbert modular eigenform 
over $K = \Q(\sqrt{5})$ of level $\Gamma_0(2\sqrt{5})$ and weight $(2,4)$. This gives rise to a motive $M$ of the desired form
by a procedure described in \cite{BR93} (which gives a motive over $K$) followed by a base change from $K$ to $\Q$.
Moment statistics for the motive $M$ over $\Q$ are listed in Table~\ref{table:hmf}, along with the corresponding moments for $G=N(G_{3,3})$. Due to computational limitations of \cite{Magma}, we were only able to compute $a_1$, and we were forced to limit the prime bound to $2^{14}$, limiting the quality of the numerical evidence. However, note that $M_4[a_1]$ appears to be converging quite rapidly to 5, and that this value occurs for no groups in Table~\ref{table:STgroups} other than $N(G_{3,3})$.
\end{example}

The motive in Example~\ref{example:hmf} is somewhat hard to write down explicitly. However, one expects that a generic example of this form should give the same Sato-Tate group, and there exist other examples where the motive appears much more explicitly.
\begin{example}
Define the two-variable Chebyshev polynomial
\[
P(x,y) = x^5 + y^5 -5xy(x^2+y^2) +5xy(x+y) +5(x^2+y^2) -5(x+y).
\]
Form the affine threefold
\[
\operatorname{Spec} \Q[x_1,x_2,x_3,x_4]/(P(x_1,x_2) - P(x_3,x_4)),
\]
then take the Zariski closure in $\mathbb{P}^4_\Q$. It was observed by
Consani-Scholten \cite{CS01} that the resulting threefold has 120 ordinary double points and no other singularities.
Blow up these double points to obtain a smooth threefold, then take middle cohomology to obtain a motive $M$.

It was conjectured by Consani-Scholten 
and proved by Dieulefait-Pacetti-Sch\"utt \cite{DPS12} that this is an example of a nonrigid
modular Calabi-Yau threefold. More precisely, the $L$-function of $M$ coincides with that of a certain Hilbert newform over $K = \Q(\sqrt{5})$ of level $\Gamma_0(30)$ (or rather its base change from $K$ to $\Q$).
\end{example}

\subsection{Other Hecke characters}

So far we have only considered Hecke characters over quadratic fields. However, algebraic Hecke characters over larger fields also correspond to motives, as described in \cite{Sc88}. We 
have seen one instance of this in another guise in
Example~\ref{example:Fac}. It is tempting to try to realize $F_{a,b,c}$ using a variant of that example; this turns out to be possible for motives with coefficients in a real quadratic field, but it remains unclear whether rational coefficients can be achieved.

\begin{example} \label{example:hecke_char}
Consider the number field $K = \Q[\alpha]/(\alpha^4-2\alpha^3+5\alpha^2-4\alpha+2)$,
labeled
\href{http://www.lmfdb.org/NumberField/4.0.1088.2}{4.0.1088.2}
in \cite{LMFDB};
this is a CM field of class number 1 whose
Galois group is the dihedral group of order 8.
Let $\fp$ be the unique (ramified) prime of norm 17.
There is then a unique algebraic Hecke character $\psi$ of conductor $\fp$ 
and infinite type $(3,0), (1,2)$.
The resulting motive $M$ is defined over $\Q$ but has coefficients
in $\Q(\sqrt{17})$; it is thus not covered by our classification.
Nonetheless, one can compute $L$-polynomial coefficients in \cite{Magma} and observe good agreement with moment statistics for the group $F_{a,b,c}$.
\end{example}

\begin{remark} \label{remark:abelian surface groups}
One can construct similar examples of infinite type $(1,0), (1,0)$.
One thus obtains motives with the Hodge numbers of an abelian surface,
but having Sato-Tate group $F_{a,b,c}$ which is shown not to occur for abelian surfaces in \cite{FKRS12}. 
In particular, the three groups appearing in the group-theoretic classification of \cite{FKRS12} which are not realized by abelian surfaces appear to be realized by motives with nonrational coefficients.
\end{remark}

\begin{remark}
For any example constructed from Hecke characters as above, the connected part of the Sato-Tate group should be a torus. If so, one can prove equidistribution using the work of Johansson \cite{Joh14}.
\end{remark}

\section{Moment statistics}\label{section: tables}
This section lists moment statistics for the various motives constructed in the previous three sections.
In each of the tables that follow, the column $n$ indicates the norm bound $B=2^n$ on the degree 1 primes $\fp$ of $K$ for which $L$-polynomials~$L_\fp(T)$ were computed.  The remaining columns list various moment statistics~$M_n[a_i]$ of the normalized $L$-polynomial coefficients $a_1$ and $a_2$.  Following each example, the corresponding moments of the candidate Sato-Tate group $G$ are listed for comparison.

\begin{landscape}
\setlength{\extrarowheight}{-0.48mm}
\setlength{\tabcolsep}{1.5mm}
\begin{table}
\begin{center}
\begin{tabular}{|r|rrrrrr|rrrrrrr|}\hline\bigstrut[t]
\hspace{3mm}$n$&$M_2[a_1]$&$M_4[a_1]$&$M_6[a_1]$&$M_8[a_1]$&$M_{10}[a_1]$&$M_{12}[a_1]$&$M_1[a_2]$&
$M_2[a_2]$&$M_3[a_2]$&$M_4[a_2]$&$M_5[a_2]$&$M_6[a_2]$&$M_7[a_2]$\\
\hline\bigstrut[t]
&\multicolumn{13}{c|}{$M=M_1\oplus M_2$ over $K=\Q(i,\omega)$ with $f_1=$ \href{http://www.lmfdb.org/ModularForm/GL2/Q/holomorphic/32/2/1/a/}{32.2.1a} and $f_2=$ \href{http://www.lmfdb.org/ModularForm/GL2/Q/holomorphic/9/4/1/a/}{9.4.1a} $\qquad$(Example \ref{example:mfsum_F})}\phantomsection\label{table:mfsum_F}\\
\hline\bigstrut[t]
12 & 3.848 & 32.096 & 329.646 & 3772.44 & 46139.8 & 589946 & 2.025 & 7.562 & 28.615 & 125.379 & 573.450 & 2761.95 & 13699.7\\
16 & 4.068 & 36.349 & 399.331 & 4828.99 & 61695.2 & 816810 & 2.043 & 8.062 & 32.259 & 148.124 & 707.888 & 3533.35 & 18071.1\\
20 & 3.977 & 35.643 & 394.090 & 4803.74 & 61964.3 & 829450 & 1.991 & 7.958 & 31.691 & 146.135 & 700.526 & 3514.70 & 18079.7\\
24 & 3.984 & 35.765 & 396.577 & 4849.68 & 62756.6 & 842562 & 1.994 & 7.966 & 31.802 & 146.893 & 705.735 & 3548.37 & 18292.4\\
28 & 3.999 & 35.978 & 399.623 & 4893.66 & 63399.1 & 852060 & 2.000 & 7.997 & 31.982 & 147.884 & 711.270 & 3579.48 & 18468.2\\
32 & 4.000 & 36.005 & 400.061 & 4900.72 & 63512.0 & 853854 & 2.000 & 8.001 & 32.004 & 148.022 & 712.107 & 3584.53 & 18498.5\\
\hline
\end{tabular}
\end{center}
\end{table}

\begin{table}
\begin{center}
\begin{tabular}{|r|rrrrrr|rrrrrrr|}\hline\bigstrut[t]
\hspace{3mm}$n$&$M_2[a_1]$&$M_4[a_1]$&$M_6[a_1]$&$M_8[a_1]$&$M_{10}[a_1]$&$M_{12}[a_1]$&$M_1[a_2]$&
$M_2[a_2]$&$M_3[a_2]$&$M_4[a_2]$&$M_5[a_2]$&$M_6[a_2]$&$M_7[a_2]$\\
\hline\bigstrut[t]
&\multicolumn{13}{c|}{$M=M_1\oplus M_2$ over $K=\Q(\omega)$ with $f_1=$ \href{http://www.lmfdb.org/ModularForm/GL2/Q/holomorphic/11/2/1/a/}{11.2.1a} and $f_2=$ \href{http://www.lmfdb.org/ModularForm/GL2/Q/holomorphic/9/4/1/a/}{9.4.1a} $\qquad$(Example \ref{example:mfsum_G13})}\phantomsection\label{table:mfsum_G13}\\
\hline\bigstrut[t]
12 & 2.831 & 18.700 & 162.998 & 1653.35 & 18513.5 & 222275 & 1.955 & 5.772 & 18.845 & 70.897 & 289.489 & 1269.72 & 5856.2\\
16 & 2.912 & 18.900 & 162.111 & 1606.02 & 17364.6 & 199045 & 1.967 & 5.823 & 19.073 & 71.426 & 289.045 & 1247.86 & 5632.4\\
20 & 2.990 & 19.881 & 173.341 & 1739.61 & 19040.2 & 221046 & 1.997 & 5.982 & 19.902 & 75.459 & 308.989 & 1346.94 & 6134.4\\
24 & 2.999 & 19.981 & 174.808 & 1762.12 & 19385.8 & 226341 & 2.000 & 5.997 & 19.986 & 75.935 & 311.725 & 1362.83 & 6227.0\\
28 & 2.999 & 19.991 & 174.872 & 1761.94 & 19371.8 & 226024 & 2.000 & 5.999 & 19.992 & 75.955 & 311.735 & 1362.45 & 6223.0\\
32 & 3.000 & 19.998 & 174.956 & 1763.31 & 19393.8 & 226369 & 2.000 & 6.000 & 19.998 & 75.987 & 311.919 & 1363.52 & 6229.3\\
\hline
\end{tabular}
\end{center}
\end{table}

\begin{table}
\begin{center}
\begin{tabular}{|r|rrrrrr|rrrrrrr|}\hline\bigstrut[t]
\hspace{3mm}$n$&$M_2[a_1]$&$M_4[a_1]$&$M_6[a_1]$&$M_8[a_1]$&$M_{10}[a_1]$&$M_{12}[a_1]$&$M_1[a_2]$&
$M_2[a_2]$&$M_3[a_2]$&$M_4[a_2]$&$M_5[a_2]$&$M_6[a_2]$&$M_7[a_2]$\\
\hline\bigstrut[t]
&\multicolumn{13}{c|}{$M=M_1\oplus M_2$ over $K=\Q(\omega)$ with $f_1=$ \href{http://www.lmfdb.org/ModularForm/GL2/Q/holomorphic/27/2/1/a/}{27.2.1a} and $f_2=$ \href{http://www.lmfdb.org/ModularForm/GL2/Q/holomorphic/9/4/1/a/}{9.4.1a} $\qquad$(Example \ref{example:mfsum_C1})}\phantomsection\label{table:mfsum_C1}\\
\hline\bigstrut[t]
12 & 3.724 & 39.781 & 517.581 & 7140.47 & 101446.7 & 1467990 & 1.908 & 7.388 & 34.459 & 175.801 & 935.697 & 5096.06 & 28152.9\\
16 & 3.968 & 43.330 & 567.235 & 7865.66 & 112441.4 & 1639111 & 1.991 & 7.926 & 37.448 & 192.185 & 1026.865 & 5613.14 & 31134.8\\
20 & 3.991 & 43.809 & 576.712 & 8037.78 & 115425.6 & 1689536 & 1.997 & 7.976 & 37.838 & 194.976 & 1045.677 & 5735.41 & 31912.2\\
24 & 3.995 & 43.912 & 578.580 & 8069.13 & 115936.1 & 1697732 & 1.998 & 7.988 & 37.926 & 195.553 & 1049.301 & 5757.72 & 32047.9\\
28 & 4.000 & 43.991 & 579.834 & 8089.14 & 116256.1 & 1702843 & 2.000 & 7.999 & 37.993 & 195.950 & 1051.676 & 5771.97 & 32133.4\\
32 & 4.000 & 43.998 & 579.954 & 8091.18 & 116289.9 & 1703399 & 2.000 & 8.000 & 37.998 & 195.986 & 1051.909 & 5773.42 & 32142.3\\
\hline
\end{tabular}
\end{center}
\end{table}

\begin{table}
\begin{center}
\begin{tabular}{|r|rrrrrr|rrrrrrr|}\hline\bigstrut[t]
\hspace{3mm}$n$&$M_2[a_1]$&$M_4[a_1]$&$M_6[a_1]$&$M_8[a_1]$&$M_{10}[a_1]$&$M_{12}[a_1]$&$M_1[a_2]$&
$M_2[a_2]$&$M_3[a_2]$&$M_4[a_2]$&$M_5[a_2]$&$M_6[a_2]$&$M_7[a_2]$\\
\hline\bigstrut[t]
&\multicolumn{13}{c|}{$M=M_1\oplus M_2$ over $K=\Q$ with $f_1=$ \href{http://www.lmfdb.org/ModularForm/GL2/Q/holomorphic/27/2/1/a/}{27.2.1a} and $f_2=f_{\psi^3}\otimes\chi_4$ $\qquad$(Example \ref{example:mfsum_C2})}\phantomsection\label{table:mfsum_JC2}\\
\hline\bigstrut[t]
12 & 1.951 & 16.715 & 176.479 & 2086.97 & 26661.8 & 359633 & 2.012 & 5.901 & 19.029 & 74.968 & 327.057 & 1561.28 & 7842.5\\
16 & 1.944 & 17.268 & 189.198 & 2309.55 & 30284.2 & 417223 & 1.979 & 5.891 & 19.382 & 78.508 & 351.613 & 1718.58 & 8806.0\\
20 & 1.998 & 17.951 & 199.147 & 2463.88 & 32722.2 & 455862 & 2.001 & 5.994 & 19.964 & 81.743 & 370.402 & 1829.06 & 9461.4\\
24 & 1.998 & 17.967 & 199.398 & 2467.78 & 32783.7 & 456826 & 2.000 & 5.997 & 19.975 & 81.822 & 370.853 & 1831.81 & 9477.8\\
28 & 2.000 & 17.994 & 199.879 & 2475.91 & 32917.0 & 458980 & 2.000 & 5.999 & 19.995 & 81.965 & 371.768 & 1837.53 & 9512.8\\
32 & 2.000 & 17.997 & 199.963 & 2477.45 & 32943.4 & 459423 & 2.000 & 6.000 & 19.998 & 81.988 & 371.932 & 1838.60 & 9519.7\\
\hline
\end{tabular}
\end{center}
\end{table}

\begin{table}
\begin{center}
\begin{tabular}{|r|rrrrrr|rrrrrrr|}\hline\bigstrut[t]
\hspace{3mm}$n$&$M_2[a_1]$&$M_4[a_1]$&$M_6[a_1]$&$M_8[a_1]$&$M_{10}[a_1]$&$M_{12}[a_1]$&$M_1[a_2]$&
$M_2[a_2]$&$M_3[a_2]$&$M_4[a_2]$&$M_5[a_2]$&$M_6[a_2]$&$M_7[a_2]$\\
\hline\bigstrut[t]
&\multicolumn{13}{c|}{$M=M_1\oplus M_2$ over $K=\Q(i)$ with $f_1=$ level 256 qurtic twist of \href{http://www.lmfdb.org/ModularForm/GL2/Q/holomorphic/32/2/1/a/}{32.2.1a} and $f_2=$ \href{http://www.lmfdb.org/ModularForm/GL2/Q/holomorphic/32/4/1/b/}{32.4.1b} $\qquad$(Example \ref{example:mfsum_C4})}\phantomsection\label{table:mfsum_C4}\\
\hline
12 & 3.956 & 35.366 & 385.901 & 4597.51 & 57507.0 & 741277 & 1.973 & 7.859 & 31.351 & 143.344 & 679.870 & 3357.91 & 16951.8\\
16 & 3.931 & 34.902 & 382.691 & 4623.07 & 59036.2 & 781489 & 1.968 & 7.838 & 31.057 & 142.474 & 678.992 & 3385.60 & 17299.0\\
20 & 3.983 & 35.704 & 395.127 & 4820.48 & 62212.4 & 832884 & 1.994 & 7.959 & 31.744 & 146.459 & 702.647 & 3527.38 & 18152.8\\
24 & 3.999 & 35.966 & 399.280 & 4887.18 & 63288.1 & 850214 & 2.000 & 7.998 & 31.975 & 147.796 & 710.614 & 3575.05 & 18439.9\\
28 & 3.999 & 35.980 & 399.691 & 4895.17 & 63428.5 & 852595 & 2.000 & 7.998 & 31.984 & 147.903 & 711.419 & 3580.56 & 18475.6\\
32 & 4.000 & 35.995 & 399.924 & 4898.78 & 63484.7 & 853470 & 2.000 & 7.999 & 31.996 & 147.976 & 711.855 & 3583.14 & 18490.8\\
\hline
\end{tabular}
\end{center}
\end{table}

\begin{table}
\begin{center}
\begin{tabular}{|r|rrrrrr|rrrrrrr|}\hline\bigstrut[t]
\hspace{3mm}$n$&$M_2[a_1]$&$M_4[a_1]$&$M_6[a_1]$&$M_8[a_1]$&$M_{10}[a_1]$&$M_{12}[a_1]$&$M_1[a_2]$&
$M_2[a_2]$&$M_3[a_2]$&$M_4[a_2]$&$M_5[a_2]$&$M_6[a_2]$&$M_7[a_2]$\\
\hline\bigstrut[t]
&\multicolumn{13}{c|}{$M=M_1\oplus M_2$ over $K=\Q$ with $f_1=$ \href{http://www.lmfdb.org/ModularForm/GL2/Q/holomorphic/27/2/1/a/}{27.2.1a} and $f_2=$ level 576 sextic twist of \href{http://www.lmfdb.org/ModularForm/GL2/Q/holomorphic/9/4/1/a/}{9.4.1a} (\S\ref{subsection: psi-3})
$\quad$(Example \ref{example:mfsum_C6})}\phantomsection\label{table:mfsum_JC6}\\
\hline\bigstrut[t]
12 & 1.939 & 17.084 & 187.881 & 2293.59 & 29679.4 & 398006 & 2.013 & 5.888 & 19.305 & 77.865 & 350.591 & 1711.68 & 8722.4\\
16 & 1.966 & 17.451 & 191.507 & 2318.72 & 29713.0 & 395064 & 1.988 & 5.925 & 19.553 & 79.313 & 356.215 & 1730.91 & 8763.8\\
20 & 1.990 & 17.854 & 197.629 & 2411.09 & 31114.7 & 416512 & 1.998 & 5.981 & 19.877 & 81.242 & 367.357 & 1795.75 & 9140.4\\
24 & 1.999 & 17.971 & 199.465 & 2441.08 & 31609.0 & 424636 & 2.000 & 5.997 & 19.976 & 81.837 & 370.969 & 1817.67 & 9273.9\\
28 & 2.000 & 17.999 & 199.982 & 2449.72 & 31746.9 & 426791 & 2.000 & 6.000 & 19.999 & 81.993 & 371.964 & 1823.78 & 9310.6\\
32 & 2.000 & 17.996 & 199.934 & 2449.03 & 31737.4 & 426669 & 2.000 & 5.999 & 19.996 & 81.979 & 371.878 & 1823.30 & 9308.0\\
\hline
\end{tabular}
\end{center}
\end{table}

\begin{table}
\begin{center}
\begin{tabular}{|r|rrrrrr|rrrrrrr|}\hline\bigstrut[t]
\hspace{3mm}$n$&$M_2[a_1]$&$M_4[a_1]$&$M_6[a_1]$&$M_8[a_1]$&$M_{10}[a_1]$&$M_{12}[a_1]$&$M_1[a_2]$&
$M_2[a_2]$&$M_3[a_2]$&$M_4[a_2]$&$M_5[a_2]$&$M_6[a_2]$&$M_7[a_2]$\\
\hline\bigstrut[t]
&\multicolumn{13}{c|}{$M=M_1\otimes M_2$ over $K=\Q(\omega)$ with $E_1\colon y^2=x^3-x$ and $E_2\colon y^2=x^3+1\qquad$(Example \ref{example:ecprod_F})}\phantomsection\label{table:ecprod_Fc}\\
\hline\bigstrut[t]
12 & 1.760 & 14.792 & 157.976 & 1879.81 & 23778.9 & 313000 & 0.889 & 4.483 & 13.217 & 62.888 & 280.877 & 1396.83 & 7026.1\\
16 & 1.960 & 17.338 & 189.871 & 2298.41 & 29474.3 & 392273 & 0.979 & 4.904 & 15.378 & 73.561 & 336.382 & 1690.81 & 8612.3\\
20 & 1.983 & 17.741 & 196.207 & 2393.69 & 30902.3 & 413896 & 0.995 & 4.964 & 15.781 & 75.780 & 348.983 & 1761.96 & 9017.0\\
24 & 1.997 & 17.960 & 199.397 & 2440.70 & 31606.3 & 424589 & 0.998 & 4.993 & 15.961 & 76.791 & 354.814 & 1795.18 & 9208.0\\
28 & 1.999 & 17.988 & 199.818 & 2447.25 & 31710.2 & 426250 & 1.000 & 4.998 & 15.989 & 76.939 & 355.653 & 1800.00 & 9236.5\\
32 & 2.000 & 17.999 & 199.975 & 2449.52 & 31743.6 & 426744 & 1.000 & 5.000 & 15.999 & 76.993 & 355.948 & 1801.66 & 9245.8\\
\hline
\end{tabular}
\end{center}
\end{table}

\begin{table}
\begin{center}
\begin{tabular}{|r|rrrrrr|rrrrrrr|}\hline\bigstrut[t]
\hspace{3mm}$n$&$M_2[a_1]$&$M_4[a_1]$&$M_6[a_1]$&$M_8[a_1]$&$M_{10}[a_1]$&$M_{12}[a_1]$&$M_1[a_2]$&
$M_2[a_2]$&$M_3[a_2]$&$M_4[a_2]$&$M_5[a_2]$&$M_6[a_2]$&$M_7[a_2]$\\
\hline\bigstrut[t]
&\multicolumn{13}{c|}{$M=\Sym^3 M_1$ over $K=\Q$ with $E_1\colon y^2=x^3+1\qquad$(Example \ref{example:eccube_C1})}\phantomsection\label{table:eccube_JC1}\\
\hline\bigstrut[t]
12 & 1.903 & 20.468 & 265.599 & 3654.23 & 51858.3 & 750843 & 1.978 & 5.807 & 21.758 & 98.382 & 495.853 & 2640.53 & 14459.8\\
16 & 1.966 & 21.434 & 280.637 & 3892.86 & 55670.0 & 811858 & 1.991 & 5.930 & 22.538 & 103.098 & 524.138 & 2810.12 & 15477.0\\
20 & 1.991 & 21.870 & 288.047 & 4016.19 & 57690.7 & 844625 & 1.997 & 5.982 & 22.892 & 105.374 & 538.364 & 2897.78 & 16012.6\\
24 & 1.998 & 21.957 & 289.272 & 4033.95 & 57955.0 & 848620 & 2.000 & 5.995 & 22.966 & 105.777 & 540.607 & 2910.44 & 16084.8\\
28 & 1.999 & 21.990 & 289.847 & 4043.60 & 58114.3 & 851225 & 2.000 & 5.999 & 22.992 & 105.952 & 541.712 & 2917.29 & 16126.9\\
32 & 2.000 & 21.998 & 289.966 & 4045.46 & 58143.4 & 851681 & 2.000 & 6.000 & 22.998 & 105.990 & 541.937 & 2918.62 & 16134.7\\
\hline
\end{tabular}
\end{center}

\end{table}
\begin{table}
\begin{center}
\begin{tabular}{|r|rrrrrr|rrrrrrr|}\hline\bigstrut[t]
\hspace{3mm}$n$&$M_2[a_1]$&$M_4[a_1]$&$M_6[a_1]$&$M_8[a_1]$&$M_{10}[a_1]$&$M_{12}[a_1]$&$M_1[a_2]$&
$M_2[a_2]$&$M_3[a_2]$&$M_4[a_2]$&$M_5[a_2]$&$M_6[a_2]$&$M_7[a_2]$\\
\hline\bigstrut[t]
&\multicolumn{13}{c|}{$M=M_1\otimes M_2$ over $K=\Q$ with $f_1=$ \href{http://www.lmfdb.org/ModularForm/GL2/Q/holomorphic/27/2/1/a/}{27.2.1a} and $f_2=f_{\psi^2}$ $\qquad$(Example \ref{example:mfprod_C1})}\phantomsection\label{table:mfprod_JC1}\\
\hline\bigstrut[t]
12 & 1.835 & 19.607 & 255.106 & 3519.41 & 50001.3 & 723547 & 1.955 & 5.670 & 21.041 & 94.763 & 477.416 & 2544.22 & 13941.0\\
16 & 1.977 & 21.592 & 282.664 & 3919.60 & 56031.6 & 816799 & 1.995 & 5.956 & 22.675 & 103.796 & 527.759 & 2829.24 & 15579.2\\
20 & 1.994 & 21.890 & 288.170 & 4016.29 & 57675.5 & 844222 & 1.999 & 5.987 & 22.909 & 105.430 & 538.511 & 2897.87 & 16009.8\\
24 & 1.997 & 21.949 & 289.194 & 4033.23 & 57948.8 & 848585 & 1.999 & 5.993 & 22.958 & 105.747 & 540.482 & 2909.92 & 16082.7\\
28 & 2.000 & 21.995 & 289.906 & 4044.43 & 58126.0 & 851391 & 2.000 & 5.999 & 22.996 & 105.972 & 541.820 & 2917.88 & 16130.1\\
32 & 2.000 & 21.998 & 289.971 & 4045.51 & 58143.8 & 851683 & 2.000 & 6.000 & 22.999 & 105.991 & 541.944 & 2918.65 & 16134.8\\
\hline
\end{tabular}
\end{center}
\end{table}

\begin{table}
\begin{center}
\begin{tabular}{|r|rrrrrr|rrrrrrr|}\hline\bigstrut[t]
\hspace{3mm}$n$&$M_2[a_1]$&$M_4[a_1]$&$M_6[a_1]$&$M_8[a_1]$&$M_{10}[a_1]$&$M_{12}[a_1]$&$M_1[a_2]$&
$M_2[a_2]$&$M_3[a_2]$&$M_4[a_2]$&$M_5[a_2]$&$M_6[a_2]$&$M_7[a_2]$\\
\hline\bigstrut[t]
&\multicolumn{13}{c|}{$M=M_1\otimes M_2$ over $K=\Q(\omega)$ with $f_1=$ \href{http://www.lmfdb.org/ModularForm/GL2/Q/holomorphic/36/2/1/a/}{36.2.1a} and $f_2=f_{\psi^2}$ $\qquad$(Example \ref{example:mfprod_C3})}\phantomsection\label{table:mfprod_C3}\\
\hline\bigstrut[t]
12 & 3.749 & 31.983 & 340.480 & 4006.11 & 49839.4 & 642488 & 1.951 & 7.443 & 28.658 & 128.351 & 599.744 & 2936.99 & 14745.9\\
16 & 3.964 & 35.185 & 384.363 & 4624.51 & 58836.7 & 776721 & 1.987 & 7.913 & 31.337 & 143.436 & 681.908 & 3391.67 & 17288.5\\
20 & 3.987 & 35.740 & 395.657 & 4830.42 & 62399.9 & 836748 & 1.996 & 7.968 & 31.781 & 146.639 & 703.736 & 3534.46 & 18200.1\\
24 & 3.997 & 35.955 & 399.202 & 4886.09 & 63267.7 & 850284 & 1.999 & 7.994 & 31.959 & 147.748 & 710.396 & 3573.93 & 18433.2\\
28 & 3.999 & 35.988 & 399.792 & 4896.57 & 63447.9 & 853300 & 2.000 & 7.998 & 31.990 & 147.936 & 711.604 & 3581.58 & 18481.2\\
32 & 4.000 & 35.996 & 399.937 & 4898.94 & 63486.6 & 853932 & 2.000 & 8.000 & 31.997 & 147.981 & 711.880 & 3583.25 & 18491.4\\
\hline
\end{tabular}
\end{center}
\end{table}

\begin{table}
\begin{center}
\begin{tabular}{|r|rrrrrr|rrrrrrr|}\hline\bigstrut[t]
\hspace{3mm}$n$&$M_2[a_1]$&$M_4[a_1]$&$M_6[a_1]$&$M_8[a_1]$&$M_{10}[a_1]$&$M_{12}[a_1]$&$M_1[a_2]$&
$M_2[a_2]$&$M_3[a_2]$&$M_4[a_2]$&$M_5[a_2]$&$M_6[a_2]$&$M_7[a_2]$\\
\hline\bigstrut[t]
&\multicolumn{13}{c|}{$M=M_1\otimes M_2$ over $K=\Q(\sqrt{3})$ with $f_1=$ \href{http://www.lmfdb.org/ModularForm/GL2/Q/holomorphic/32/2/1/a/}{32.2.1a} and $f_2=f_{\psi^2}$ $\qquad$(Example \ref{example:mfprod_F})}\phantomsection\label{table:mfprod_Fab}\\
\hline\bigstrut[t]
12 & 1.868 & 16.641 & 184.352 & 2254.74 & 29192.6 & 392382 & 1.930 & 5.692 & 18.744 & 76.066 & 343.596 & 1678.10 & 8546.9\\
16 & 1.967 & 17.762 & 197.624 & 2419.20 & 31295.9 & 419767 & 1.976 & 5.921 & 19.699 & 80.773 & 366.398 & 1796.05 & 9162.7\\
20 & 2.004 & 18.054 & 200.477 & 2453.21 & 31758.5 & 426537 & 1.999 & 6.010 & 20.028 & 82.165 & 372.599 & 1826.24 & 9318.1\\
24 & 1.997 & 17.956 & 199.351 & 2440.60 & 31614.3 & 424841 & 1.999 & 5.994 & 19.960 & 81.776 & 370.763 & 1817.17 & 9274.0\\
28 & 1.999 & 17.980 & 199.673 & 2444.93 & 31674.1 & 425690 & 2.000 & 5.998 & 19.983 & 81.896 & 371.379 & 1820.35 & 9290.7\\
32 & 2.000 & 17.996 & 199.943 & 2449.13 & 31738.3 & 426669 & 2.000 & 5.999 & 19.997 & 81.981 & 371.892 & 1823.37 & 9308.3\\
\hline
\end{tabular}
\end{center}
\end{table}

\begin{table}
\begin{center}
\begin{tabular}{|r|rrrrrr|rrrrrrr|}\hline\bigstrut[t]
\hspace{3mm}$n$&$M_2[a_1]$&$M_4[a_1]$&$M_6[a_1]$&$M_8[a_1]$&$M_{10}[a_1]$&$M_{12}[a_1]$&$M_1[a_2]$&
$M_2[a_2]$&$M_3[a_2]$&$M_4[a_2]$&$M_5[a_2]$&$M_6[a_2]$&$M_7[a_2]$\\
\hline\bigstrut[t]
&\multicolumn{13}{c|}{$M=M_1\otimes M_2$ over $K=\Q(\omega)$ with $f_1=$ \href{http://www.lmfdb.org/ModularForm/GL2/Q/holomorphic/11/2/1/a/}{11.2.1a} and $f_2=f_{\psi^2}$ $\qquad$(Example \ref{example:mfprod_U2})}\phantomsection\label{table:mfprod_U2}\\
\hline\bigstrut[t]
12 & 2.050 & 12.696 & 106.126 & 1044.08 & 11462.6 & 135998 & 0.924 & 4.159 & 11.354 & 46.222 & 181.157 & 799.21 & 3637.5\\
16 & 1.962 & 11.714 & 98.200 & 972.33 & 10623.3 & 123788 & 0.982 & 3.930 & 10.743 & 43.127 & 169.178 & 744.94 & 3376.7\\
20 & 1.991 & 11.920 & 99.255 & 972.75 & 10517.0 & 121449 & 0.997 & 3.972 & 10.915 & 43.646 & 170.614 & 746.31 & 3360.0\\
24 & 1.997 & 11.962 & 99.543 & 974.22 & 10507.4 & 120917 & 0.999 & 3.993 & 10.968 & 43.835 & 171.159 & 747.59 & 3359.6\\
28 & 1.999 & 11.995 & 99.997 & 980.52 & 10596.9 & 122213 & 1.000 & 3.998 & 10.996 & 43.989 & 172.006 & 752.29 & 3385.8\\
32 & 2.000 & 11.999 & 99.995 & 979.98 & 10584.6 & 121987 & 1.000 & 4.000 & 10.999 & 43.997 & 171.988 & 751.97 & 3383.0\\
\hline
\end{tabular}
\end{center}
\end{table}

\begin{table}
\begin{center}
\begin{tabular}{|r|rrrrrr|rrrrrrr|}
\hline\bigstrut[t]
$n$&$M_2[a_1]$&$M_4[a_1]$&$M_6[a_1]$&$M_8[a_1]$&$M_{10}[a_1]$&$M_{12}[a_1]$&$M_1[a_2]$&
$M_2[a_2]$&$M_3[a_2]$&$M_4[a_2]$&$M_5[a_2]$&$M_6[a_2]$&$M_7[a_2]$\\
\hline\bigstrut[t]
&\multicolumn{13}{c|}{$M$ is the motive arising from the quintic threefold \eqref{eq:quintic} with $t=-5$  $\qquad$(Example \ref{example:USp4})}\phantomsection\label{table:USp4}\\
\hline\bigstrut[t]
10 & 1.038 & 2.956 & 11.783 & 56.21 & 304.9 & 1800 & 0.999 & 2.002 & 3.826 & 9.221 & 22.507 & 61.02 & 170.2 \\
13 & 0.974 & 2.833 & 12.281 & 65.88 & 404.0 & 2717 & 0.989 & 1.969 & 3.871 & 9.498 & 24.295 & 69.34 & 207.7 \\
16 & 0.985 & 2.984 & 14.371 & 89.80 & 659.0 & 5372 & & & & & & & \\
18 & 0.986 & 2.916 & 13.465 & 80.05 & 560.1 & 4384 & & & & & & & \\
20 & 1.001 & 3.021 & 14.205 & 85.54 & 603.1 & 4740 & & & & & & & \\
22 & 0.999 & 2.996 & 13.968 & 83.68 & 590.3 & 4673 & & & & & & & \\
\hline
\end{tabular}
\end{center}
\end{table}

\end{landscape}

\section*{Acknowledgments}
Thanks to Josep Gonz\'alez, Joan-C. Lario, Fernando Rodriguez Villegas,
and Mark Watkins for helpful discussions. Thanks to Jean-Pierre Serre for suggesting the construction of \S\ref{section: sum}. 
%Fit\'e received financial support from the German Research council, via CRC 701.
%Kedlaya was supported by NSF (grant DMS-1101343) and
%UCSD (Stefan E. Warschawski professorship). 
%Sutherland was supported by NSF (grant DMS-1115455).

\end{document}